%% file: part_1.tex
\documentclass[a4paper,12pt]{amsart}

\usepackage[T1]{fontenc}
\usepackage[utf8]{inputenc}

\usepackage{amsmath}
\usepackage{amsthm}
\usepackage{amssymb}
\usepackage{mathtools}
\usepackage{enumitem}

\usepackage{mathrsfs}
\usepackage{stmaryrd}

\usepackage{nicefrac}

\usepackage[sort&compress,comma,authoryear]{natbib}

\usepackage[margin=3.25cm]{geometry}

\usepackage[hypertexnames=false]{hyperref}
\usepackage{xspace}
\usepackage{comment}

\usepackage{xcolor}
\hypersetup{
    colorlinks,
    linkcolor={red!90!black},
    citecolor={green!70!black},
    urlcolor={blue!80!black}
}

\usepackage{tikz}
\usetikzlibrary{shapes.geometric}

\usepackage[format=hang,font={footnotesize}]{caption}

\newcommand\Betw{B} 

\DeclareMathOperator{\At}{\mathrm{At}}

\newcommand\Bmod[1]{\mathop{\langle\Betw\rangle_{#1}}} 

\newcommand{\Bleqs}{B_{\leqslant}} 

\newcommand*{\QED}{\null\nobreak\hfill\ensuremath{\mathord{\dashv}}}

\usepackage[all]{xy}

\newcommand{\K}{\ensuremath{\mathcal{K}}\xspace}
\newcommand{\klam}[1]{\ensuremath{\langle #1 \rangle}}

\newcommand{\poss}[1]{\ensuremath{\langle { #1 \rangle}}}
\newcommand{\suff}[1]{\ensuremath{[[ #1 ]]}}
\newcommand{\set}[1]{\ensuremath{\{#1\}}}

\newcommand{\z}{\emptyset}
\newcommand{\tand}{\text{ and }}

\newcommand{\Va}{\ensuremath{\mathbf{V}}\xspace}
\newcommand{\win}{\boxbar}
\newcommand{\ua}[1]{\ensuremath{\mathop{\uparrow}#1}}

\newcommand{\two}{\ensuremath{\mathord{\mathbf{2}}}\xspace}

\newcommand{\into}{\hookrightarrow}

\newcommand{\Kt}{\ensuremath{K\tilde{~}}\xspace}
\newcommand{\tor}{\text{ or }}

\newcommand{\card}[1]{\vert #1 \vert}

\newcommand{\br}{betweenness relation\xspace}

\newcommand{\df}{\defeq}

\newcommand{\tiff}{if and only if\xspace}
\newcommand{\aright}{($\rarrow$)\xspace}
\newcommand{\aleft}{($\larrow$)\xspace}
\newcommand{\tbc}{\par\medskip\textcolor{blue}{{\sf To be continued.}}}

\DeclareMathOperator{\ultV}{\mathscr{V}}

\DeclareMathOperator{\Cbtw}{\mathbf{Cbtw}}
\DeclareMathOperator{\Abtw}{\mathbf{Abtw}}

\newcommand{\Cm}{\ensuremath{\operatorname{{\mathsf{Cm}}}}}
\newcommand{\Cf}{\ensuremath{\operatorname{{\mathsf{Cf}}}}}
\newcommand{\Em}{\ensuremath{\operatorname{{\mathsf{Em}}}}}
\newcommand{\Ce}{\ensuremath{\operatorname{{\mathsf{Ce}}}}}

 \parskip=\smallskipamount

{\par\medskip\color{blue}{ID: \  }}{\hfill{$\square$}\color{black}\par\medskip}

{\color{magenta}{RG: \  }}{\hfill{$\square$}\color{black}\par\medskip}

{\color{teal}{PM: \  }}{\hfill{$\square$}\color{black}\par\medskip}

\input{bbm-without-txfonts}
\input{macros.tex}

\theoremstyle{definition}

\date{}

\title{Betweenness algebras}
\author{Ivo D\"untsch, Rafa\l\ Gruszczy\'nski, Paula Mench\'on}

\address{Ivo D\"untsch\\
Computer Science Department\\
Brock University\\	
St Catharines, Ontario\\
Canada\\
\textsc{Orcid:} 0000-0001-8907-2382}

\email{duentsch@brocku.ca}

\address{Rafa\l\ Gruszczy\'nski\\
Department of Logic\\
Nicolaus Copernicus University in Toru\'n\\
Poland\\
\textsc{Orcid:} 0000-0002-3379-0577}
\email{gruszka@umk.pl}

\urladdr{www.umk.pl/\textasciitilde gruszka}

\address{Paula Mench\'on\\
Department of Logic\\
Nicolaus Copernicus University in\linebreak Toru\'n\\
Poland\\
\textsc{Orcid:} 0000-0002-9395-107X}
\email{paula.menchon@v.umk.pl}

\begin{document}

\maketitle

\begin{abstract}
    We introduce and study a class of \emph{betweenness algebras}---Boolean algebras with binary operators, closely related to ternary frames with a betweenness relation. From various axioms for betweenness, we chose those that are most common, which makes our work applicable to a wide range of betweenness structures studied in the literature. On the algebraic side, we work with two operators of \emph{possibility} and of \emph{sufficiency}.

    \medskip


\medskip

\end{abstract}

\section{Introduction}

\begin{sloppypar}
Betweenness relations---well-known from geometry---are probably the most deeply investigated ternary relations in logic and mathematics. The origin of the studies can be traced back at least to the works of Huntington and Kline \citeyearpar{Huntington-ANSOPFBWPOCI,Huntington-et-al-SOIPFB}, through the seminal contributions of Tarski \citep{Tarski-Givant-TSG}, up to the results of \citet{Altwegg-ZADTGM}, \citet{Sholander-TLOAB}, \citet{Dvelmeyer-et-al-ACOOSALVBR} and \citet{du_between}.
\end{sloppypar}

At least as early as in the seminal papers of \citet{jt51}, the connection between $n+1$-ary relations and $n$-ary operators on Boolean algebras was established in the form of J\'onsson-Tarski duality for Boolean algebras with operators. The developed techniques turned out to be particularly successful in the study of modal logics and their algebraic semantics.

The abstract approach of \citeauthor{jt51} can be made concrete by focusing on a relation (or relations) of particular choice, betweenness in the case of the approach from our work. As observed by \citet{vanBenthem-MLS}, a ternary betweenness relation $B$ gives rise to the binary modal operator $\Bmod{}$ whose relational semantics is given by the following condition:
\[
x\Vdash\Bmod{}(\varphi,\psi)\iffdef(\exists\,{y,z\in U})\,(\Betw(y,x,z)\ \text{and}\ y\Vdash\varphi\ \text{and}\ z\Vdash\psi)\,,
\]
where $\Betw(y,x,z)$ is interpreted as \emph{point $x$ is between points $y$ and $z$}. Our intention is to investigate the algebraic properties of $\Bmod{}$ within the framework of Boolean algebras with operators.

In Section~\ref{sec:bao} we recall basic facts about \emph{possibility} and \emph{sufficiency} operators which will find their application in the sequel. In Section~\ref{sec:betweenness} we commit ourselves to a~particular notion of \emph{betweenness} by choosing what we see as the core axioms for the reflexive version of the geometric relation. Section~\ref{sec:non-definability} justifies our approach via possibility and sufficiency operators as we show that the class of the so-called \emph{betweenness frames} is neither modally nor sufficiency axiomatic. In Section~\ref{sec:cmcf} rudimentary facts about complex algebras of betweenness frames are established. These serve as  <<r\^ole models>> for first-order axioms for what we call \emph{betweenness algebras} in Section~\ref{sec:BetAlg}. In Section~\ref{sec:repAbtw} we prove that the class of betweenness algebras is closed under canonical extensions, and in Section~\ref{sec:BetAlg-BetComplAlg} we examine the connections between complex and abstract betweenness algebras, in~particular, we show where they differ. Section~\ref{sec:conclusion}---which rounds off the paper---points to possible extensions of our work and open problems that we are going to tackle in further installments to the work we present herein.

In the paper, $-X$ is the set-theoretical complement of $X$ with respect to a fixed domain. For an element $x$ of an ordered set $\klam{U,\leq}$, $\upop x\defeq\{y\in U \mid x\leq y\}$. We assume that unary operators always bind stronger than $n$-ary operators with $n\geqslant 2$. Throughout, algebras are assumed to be non-trivial (i.e., have at least two elements), and $\klam{A,+,\cdot, \zero, \one}$ is a non-trivial Boolean algebra (BA). To avoid notational cluttering we shall usually identify Boolean algebras with their universe; in particular, $2^U$ denotes the power set algebra of~$U$. Furthermore, $\At(A)$ is the set of atoms of $A$ and $\Ult(A)$ the set of its ultrafilters.

Our general references are \citep{Blackburn-et-al-ML} for modal logic, \citep{kop89} for Boolean algebras and \citep{bs_ua} for universal algebra.

\section{Boolean algebras with operators}\label{sec:bao}

Boolean algebras with operators (BAOs) were introduced by \cite{jt51}, arising from Tarski's work on relation algebras. It turned out later that BAOs were intimately connected to the semantics of modal logics; for details the reader is encouraged to consult \cite[Section 3.3]{gold03}. An extensive treatment of duality theories for BAOs was presented by \citet{giv14}.

In this section we shall outline the properties of BAOs as we need them, and augment them by a theory of sufficiency algebras.

\begin{definition}\label{def: modop}
    A mapping  $f\colon A^n\to A$ (where $n\geqslant 1$) is an \emph{$n$-ary possibility operator} \tiff it satisfies the following conditions:
    \begin{enumerate}[label=(P\arabic*),itemsep=0pt]
        \item If there is $i$ such that $1\leqslant i\leqslant n$ and $x_i=\zero$, then $f(x_1,\ldots,x_n)=\zero$ (normality).
        \item If $\klam{x_1, \ldots, x_n}$ and $\klam{y_1, \ldots, y_n}$ are $n$-termed sequences in $A$ such that $x_i = y_i$ for all $i \neq k$, then $f(x_1,\ldots, x_k, \ldots,x_n)+f(y_1, \ldots, y_k, \ldots,y_n)=f(x_1, \ldots,  x_k+y_k, \ldots, x_n)$ (additivity). \qed
    \end{enumerate}
\end{definition}
Observe that a possibility operator is monotone in each argument.

Suppose that $1 \leqslant n$ and $f\colon  A^n \to A$ is a possibility operator on $A$. The \emph{possibility canonical frame} of $\frA \defeq \klam{A,f}$ is the structure $\Cf^p(\frA) \defeq \klam{\Ult(A), Q_f}$ where $Q_f$ is the $n+1$-ary relation on $\Ult(A)$ defined by
\begin{gather}\label{def: cfp}
Q_f(\ult_1, \ldots, \ult_{n+1}) \iffdef f[\ult_1 \times \ldots \times \ult_n] \subseteq \ult_{n+1}.
\end{gather}
The choice of the $n+1^{\text{st}}$ entry as special is arbitrary; any other index would have done. Indeed, when considering the binary operators on betweenness algebras arising from ternary relations we shall use the middle component.

Suppose that $\frF \defeq \klam{U,Q}$ is a frame where $Q$ is an $(n+1)$-ary relation on $U$. We define an $n$-ary operator $\poss{Q}\colon \left(2^U\right)^n \to 2^U$ by
\begin{gather}\label{def: possR}
\poss{Q}(X_1, \ldots, X_n) \defeq \set{u \in U\mid (\exists x_1 \in X_1 \ldots \exists x_n \in X_n)\,Q(x_1, \ldots, x_n,u)}.
\end{gather}
It is well known that $\poss{Q}$ is a complete possibility operator \citep{jt51}. The structure $\klam{2^U, \poss{Q}}$ is called the \emph{full possibility complex algebra over $\frF$}, denoted by $\Cm^p(\frF)$. Each subalgebra is called a \emph{possibility complex algebra over $\frF$}. If $\frF$ is understood we shall just speak of possibility complex algebras.

The following result is decisive for the theory of BAOs and generalizes Stone's theorem of representing Boolean algebras:

\begin{theorem}[{\citealp[Theorem 3.10]{jt51}}]\label{thm:repP}
If $\frA \defeq \klam{A,f}$ is a Boolean algebra with an $n$-ary possibility operator $f$, then the Stone map $h\colon h \into 2^{\Ult(A)}$, defined by $h(x) \defeq \set{\ult \in \Ult(A)\mid x \in\ult}$ is an embedding of $\frA$ into  $\Cm^p(\Cf^p(\frA))$.
\end{theorem}

The algebra $\Cm^p(\Cf^p(\frA))$ is called the \emph{possibility canonical extension} of $\frA$, denoted by $\Em^{p}(\frA)$. For details of the origin and theory of BAOs see the survey by \citet{jon93}.

Starting with an $n+1$-ary frame $\frF \defeq \klam{U,R}$ there is also a representation theorem:

\begin{theorem}\label{thm:appFrameRep}
$\frF$ can be embedded into the canonical frame of its full possibility complex algebra.
\end{theorem}
\begin{proof}
This is a generalization of \cite[Theorem 3.2.7]{orr15}. Define $k\colon \frF \to \Ult(\Cm^p(\frF))$ by $k(x) \defeq \ua{\set{x}}$, that is, $k(x)$ is the principal filter of $2^U$ generated by $\set{x}$; clearly, $k$ is injective. We define $\poss{R}$ as in \eqref{def: possR}, and $Q_{\poss{R}}$ as in \eqref{def: cfp}.

Let $x_1, \ldots, x_{n+1} \in U$. Then:
{\small
\begin{align*}
Q_{\poss{R}}(k(x_1), \ldots, k(x_{n+1}))
&\iff \poss{R}\left[\ua{\set{x_1}}, \ldots, \ua{\set{x_n}}\right] \subseteq \ua{\set{x_{n+1}}} \\
&\iff (\forall X_i)\,(x_i \in X_i, 1 \leq i \leq n \rarrow \poss{R}(X_1, \ldots X_n) \in \ua{\set{x_{n+1}}}) \\
&\iff (\forall X_i)\,(x_i \in X_i, 1 \leq i \leq n \rarrow x_{n+1} \in \poss{R}(X_1, \ldots X_n)) \\
&\iff (\forall X_i)(x_i \in X_i, 1 \leq i \leq n \rarrow\\
&\phantom{(\forall X_i)(x_i \in X_i,}(\exists u_i \in X_i)R(u_1, \ldots, u_n,x_{n+1}))\,.
\end{align*}
}
If $R(x_1, \ldots, x_{k+1})$, we choose $u_i \df x_i$; this shows that $Q_{\poss{R}}(k(x_1), \ldots, k(x_{n+1}))$. Conversely, if $Q_{\poss{R}}(k(x_1), \ldots, k(x_{n+1}))$, choosing $X_i \defeq \set{x_i}$ shows that the tuple $\klam{x_1, \ldots, x_{k+1}}$ is in $R$.
\end{proof}

The relational semantics of classical modal logics is limited in expression since it can talk about (some) properties of a binary relation $R$ but not about properties of $-R$. A~``sufficiency'' counterpart of the modal necessity operator $\Box$ was independently suggested by \citet{Humberstone-IW} with $\blacksquare$ and \citet{gpt87} with $\win$.\footnote{
See also \citep{Goldblatt-SAOO} and \citep{vanBenthem-MDL} for related work.
}
The algebraic properties of such an operator and representation properties were investigated in a sequence of papers by \cite{do_mixalg,do_baro} and \citet*{dot_mixed}.

\begin{definition}\label{def: suffop}
 A mapping $g\colon A^n\to A$ is an \emph{$n$-ary sufficiency operator} \tiff $g$ meets the following two constraints:
     \begin{enumerate}[label=(S\arabic*),itemsep=0pt]
        \item If there is $i$ such that $1\leqslant i\leqslant n$ and $x_i=\zero$, then $g(x_1,\ldots,x_n)=\one$ (co-normality).
        \item If $\klam{x_1, \ldots, x_n}$ and $\klam{y_1, \ldots, y_n}$ are $n$-termed sequences in $A$ such that $x_i = y_i$ for all $i \neq k$, then $g(x_1,\ldots, x_k, \ldots,x_n) \cdot g(y_1, \ldots, y_k, \ldots,y_n)=g(x_1, \ldots,  x_k + y_k, \ldots, x_n)$ (co-additivity).
    \end{enumerate}
\end{definition}
Note that a sufficiency operator is antitone in each argument.

The pair $\frA \defeq \klam{A,g}$ is called a \emph{sufficiency algebra}. While unary possibility algebras are algebraic models of the logic $\mathsf{K}$, the unary sufficiency algebras are algebraic models of its counterpart $\mathsf{K^*}$ \citep{teh85}. The \emph{sufficiency canonical frame} is the system $\klam{\Ult(A), S}$ where $S$ is the $(n+1)$-ary relation on $\Ult(A)$ defined by
\begin{gather*}
S(\ult_1, \ldots, \ult_{n+1}) \iffdef g[\ult_1 \times \ldots \times \ult_n] \cap \ult_{n+1} \neq \z,
\end{gather*}
denoted by $\Cf^s(\frA)$. Conversely, If $\frF \defeq \klam{U,S}$ is an $(n+1)$-ary frame we define an $n$-ary operator $\suff{S}$ on $2^U$ by
\begin{gather*}
\suff{S}(X_1, \ldots, X_n) \defeq \set{u \in U\mid X_1 \times \ldots \times X_n \times \set{u} \subseteq S}.
\end{gather*}
The algebra $\klam{2^U, \suff{S}}$ is called the \emph{full sufficiency complex algebra} of $\frF$, denoted by $\Cm^s(\frF)$. Each subalgebra is called a \emph{sufficiency complex algebra over $\frF$}. If $\frF$ is understood we shall omit the reference to $\frF$. It is well known that $\suff{S}$ is a complete co-additive operator on $2^U$ \cite[Proposition 5]{do_mixalg}. In analogy to possibility algebras we have
\begin{theorem}[{\citealp{do_mixalg}}]\label{thm:repS}
\begin{enumerate}\itemsep0pt
\item If a mapping $g$ is an $n$-ary sufficiency operator on $\frA$, the Stone map $h\colon \frA \to 2^{\Ult(A)}$ is an embedding of $\frA$ into  $\Cm^s(\Cf^s(\frA))$.
\item  If $\frF \defeq \klam{U,S}$ is an $(n+1)$-ary frame the map $k\colon \frF \to \Ult(\Cm^s(\frF))$ such that $k(x) \defeq \ua{\set{x}}$ is an embedding.
\end{enumerate}
\end{theorem}
Any algebra $\frA \defeq\langle A,f,g\rangle$ such that $A$ is a BA and $f$ and $g$ are---respectively---a possibility and a sufficiency operator of the same arity will be called a \emph{Possibility--Sufficiency-algebra} (PS-algebra). Since the mappings $h$ and $k$ are the same as in Theorem  \ref{thm:repP} and Theorem \ref{thm:appFrameRep} we can define the PS--canonical frame of $\frA$ as $\klam{\Ult(A), Q_f, S_g}$ and denote it by $\Cf^{ps}(\frA)$. The algebra $\Cm^{ps}(\Cf^{ps}(\frA))$ is called the \emph{canonical extension of $\frA$}, denoted by $\Em^{ps}(\frA)$. The structure $\Cf^{ps}(\Cm^{ps}(\frF))$ is the \emph{canonical extension of $\frF$} denoted by $\Ce^{ps}(\frF)$.

From the outset, there is no connection between the possibility operator $f$ and the sufficiency operator $g$. To enhance the expressiveness of the combined corresponding logics, \citet{do_mixalg} introduced the class of \emph{mixed algebras} (MIAs) which are PS-algebras $\frA \df \klam{A,f,g}$ which satisfy the condition $Q_f = S_g$; in the unary case this is equivalent to
\begin{gather}\label{MIA}
    (\forall \ult_1,\ult_2 \in \Ult(A))
    \left(f[\ult_1] \subseteq \ult_2 \iff g[\ult_1] \cap \ult_2 \neq \z\right)\,.
\end{gather}
It was shown that the class of MIAs is not first order axiomatizable and that the canonical extension of a MIA  is isomorphic to the full complex algebra $\klam{2^U,\poss{R}, \suff{R}}$ of a frame. For an overview and examples of mixed algebras with unary operators see \cite[Section 3.6]{orr15}. Subsequently, \citet*{dot_mixed} introduced a first-order axiomatizable proper subclass of mixed algebras, called \emph{weak MIAs}, which satisfy in the unary case the axiom:
 \begin{equation}\label{wMIA}
    a \neq \zero \rarrow g(a) \leq f(a).
    \end{equation}
It turned out that the equational class generated by the weak MIAs are the algebraic models of the logic \Kt, presented by \citet{gpt87}. We shall see later that the axioms of betweenness relations can be algebraically expressed in weak MIAs, but not by possibility or sufficiency operators alone.

\section{A definition of betweenness}\label{sec:betweenness}

In the sequel, we will focus on relational systems $\langle U,B\rangle$ such that $B$ is a~ternary relation on a non-empty set~$U$. Such systems will be called \emph{3-frames}.

\begin{definition}
Let $\langle U,B\rangle$ be a 3-frame. $\Betw$ is called a \emph{betweenness relation} if it satisfies the following (universal) axioms:
\begin{gather*}
\Betw(a,a,a)\,,\tag{BT0}\label{BT0}\\
\Betw(a,b,c)\rarrow\Betw(c,b,a)\,,\tag{BT1}\label{BT1}\\
\Betw(a,b,c)\rarrow\Betw(a,a,b)\,,\tag{BT2}\label{BT2}\\
\Betw(a,b,c)\wedge\Betw(a,c,b)\rarrow b=c\,\tag{BT3}\label{BT3}\,.
\end{gather*}
Note that \eqref{BT0}--\eqref{BT2} are expanding in the sense that they require certain triples to be in a \br, while \eqref{BT3} is contracting, since it prohibits triples to be in $B$.
\QED
\end{definition}
\begin{definition}\label{def: weakbtw}
    A ternary relation $B$ is a \emph{weak betweenness} if it satisfies \eqref{BT0}--\eqref{BT2} and
    \begin{equation}\label{BTW}\tag{BTW}
    \Betw(a,b,a)\rarrow a=b\,.
\end{equation}\QED
\end{definition}

The following example shows that \eqref{BTW} is strictly weaker than \eqref{BT3}:

\begin{example}\label{ex:Wnot3}
 Set $U\defeq\{0,1,2\}$ and define
    \begin{multline*}
        \mathord{\Betw}\defeq \{\langle a,a,a\rangle\mid a\in U\}{}\cup\{\langle a,a,b\rangle\mid a,b\in U\}{}\cup{}\\
        \{\langle a,b,b\rangle\mid a,b\in U\}\cup\{\langle 0,1,2\rangle,\langle 2,1,0\rangle,\langle 0,2,1\rangle,\langle 1,2,0\rangle\}\,.\qedhere
    \end{multline*}
 Then, \eqref{BTW} is vacuously true, and $\klam{0,1,2}, \klam{0,2,1} \in B$.
\QED\end{example}
To show the difference consider the following condition:
\begin{equation}\tag{C}\label{C}
    \#\langle a,b,c\rangle\rarrow (\langle a,b,c\rangle\notin B\vee \langle a,c,b\rangle\notin B)\,.
\end{equation}
Here, $\#\klam{a,b,c}$ \tiff $\vert\set{a,b,c}\vert = 3$.

\pagebreak

\begin{proposition}\label{prop:B-is-weak-B}
Assume $B$ satisfies \eqref{BT2}. Then, $B$ satisfies \eqref{BT3} \tiff it satisfies \eqref{BTW} and \eqref{C}.
\end{proposition}
\begin{proof}
\aright  For \eqref{BTW} consider
\begin{gather*}
B(x,y,x) \overset{\eqref{BT2}}{\longrightarrow} B(x,x,y) \overset{\eqref{BT3}}{\longrightarrow} x = y.
\end{gather*}
  If $B(x,y,z)$ and $B(x,z,y)$, then $y = z$ by \eqref{BT3}, and thus, not $\#\klam{x,y,z}$.

  \aleft Suppose that $B(x,y,z)$ and $B(x,z,y)$; then, not $\#\klam{x,y,z}$ by \eqref{C}. If $y = z$, there is nothing more to show. If $x = z$, then $B(x,y,x)$, and $x = y$ by \eqref{BTW}, hence, $y = z$. Similarly, if $x = y$, then $B(x,z,x)$ and therefore $x = y = z$.
\end{proof}

\begin{definition}
$B$ is a \emph{strong betweenness} if $B$ meets the following stronger version of \eqref{BT2}:
\begin{equation}\tag{BT2s}\label{BT2s}
    B(a,a,b).
\end{equation}\QED
\end{definition}

Clearly, \eqref{BT2s} implies \eqref{BT2}. In the presence of symmetry in the form of \eqref{BT1}, the  axiom \eqref{BT2s} is equivalent to Tarski's axiom 12 from \citep{Tarski-Givant-TSG}.

The choice of the axioms is by no means arbitrary, but embodies what can be seen as the \emph{core} axioms for reflexive betweenness. Reflexivity in the form of \eqref{BT0} is equivalent to Axiom 13 of \citet{Tarski-Givant-TSG}. Symmetry, which is \eqref{BT1}, is taken as Postulate~A by \citet{Huntington-et-al-SOIPFB}\footnote{It must be said, though, that they work with the strict betweenness.} and Axiom 14 by \citeauthor{Tarski-Givant-TSG}. \eqref{BT2} in the presence of symmetry can be seen as a weakening of Tarski's reflexivity axiom (Axiom 12) and is one of the axioms for betweenness obtained from binary relations (see \citealp{Altwegg-ZADTGM}). \eqref{BT3} arises naturally in the context of---again---betweenness induced by binary relations and has a clear and natural geometric meaning. \eqref{BTW} is present in Tarski's system as Axiom 6.

For more on the motivation for the choice of \eqref{BT0}--\eqref{BT3} the reader is invited to consult \citep{du_between}.

\begin{definition}
    A pair $\frF\defeq\langle U,B\rangle$ such that $B\subseteq U^3$ and $B$ satisfies \eqref{BT0}--\eqref{BT3} will be called a \emph{betweenness frame} or just a \emph{b-frame}. If we replace \eqref{BT3} by \eqref{BTW} then $\frF$ is called a \emph{weak betweenness frame}, and in case \eqref{BT2s} is substituted for \eqref{BT2}, a~\emph{strong betweenness frame}.
\end{definition}

\section{Non-definability of betweenness relations}\label{sec:non-definability}

In the algebraic approach to betweenness, we are going to engage both possibility and sufficiency operators. This is justified by the fact that betweenness is neither possibility nor sufficiency axiomatic. We devote this section to the proofs of the aforementioned phenomena.

\subsection{Bounded and co-bounded morphisms}

The standard notion of a \emph{boun\-ded morphism} for binary relations has a natural generalization to $n$-ary ones. We restrict ourselves to $3$-frames since this is all we require.

\begin{definition}
    If $\langle U,R\rangle$ and $\langle V,S\rangle$ are 3-frames, then a mapping $f\colon U\to V$ is a \emph{bounded morphism} if
    \begin{enumerate}[label=(\arabic*),itemsep=0pt]
        \item if $R(x,y,z)$, then $S(f(x),f(y),f(z))$ (i.e., $f$ preserves $R$, i.e., satisfies the forth condition);
        \item if $S(f(w),x,y)$, then $(\exists u,v\in U)(f(u)=x\wedge f(v)=y\wedge R(w,u,v))$ (i.e., $f$ satisfies the back condition).\footnote{\citep[see p.\,140]{Blackburn-et-al-ML}}
    \end{enumerate}
    $f\colon U\to V$ is called a \emph{co-bounded morphism} if for all $x,y,z\in U$ and $t\in V$;
    \begin{enumerate}[label=(\alph*),itemsep=0pt]
        \item if $-R(x,y,z)$, then $-S(f(x),f(y),f(z))$ (i.e., $f$ preserves $-R$);
        \item if $-S(f(w),x,y)$, then $(\exists u,v\in U)(f(u)=x\wedge f(v)=y\wedge -R(w,u,v))$ (i.e., $f$~satisfies the back condition).\QED
    \end{enumerate}
\end{definition}

Since in the case of betweenness relations the middle argument plays a distinguished role, we allow ourselves to modify the definition accordingly when we need it without spelling it out explicitly.

The following two theorems are crucial for the sequel.

\begin{theorem}[{\citealp[Theorem 3]{Goldblatt-Thomason-ACIPML}}]\label{th:key-for-possibility}
    Let $\tau$ be a modal similarity type. A~first-order definable class of $\tau$ frames is possibility definable \tiff it is closed under taking bounded homomorphic images, generated subframes, disjoint unions, and reflects ultrafilter extensions.
\end{theorem}

\begin{theorem}[{\citealp[Section 5]{do_mixalg}}]\label{th:key-for-insufficiency}
    Let $\tau$ be a sufficiency similarity type. A first-order definable class of $\tau$ frames $\langle U,R\rangle$ is sufficiency definable \tiff the class of its complementary frames $\langle U,-R\rangle$ is possibility definable.
\end{theorem}

\subsection{Non-definability}

For the first of the two non-definability results (and for several examples in the sequel), we invoke relevant facts about betweenness relations obtained from binary ones. If $\langle U,R\rangle$ is a  binary frame, then it induces a <<natural>> ternary relation $B$ on $U$:
\[
B_R(x,y,z)\iffdef x\Rel y\Rel z\vee z\Rel y\Rel x\,.
\]

\noindent For these, we mention only one more definition and two basic facts:

\begin{definition}
A binary relation $\Rel$ on $U$ is called \emph{strongly antisymmetric} \tiff
\[
    x\Rel y\Rel z\Rel x\rarrow y=z.
\]
\QED\end{definition}

\begin{proposition}[{\citealp{Gruszczynski-Menchon-NOBP}}]\label{prop:B-from-SA}
    A reflexive $\Rel~\subseteq U^2$  is strongly antisymmetric if and only if $B_R$ is a betweenness relation.
\end{proposition}
\begin{corollary}\label{prop:B-from-poset}
  If $R\subseteq U^2$ is a partial order relation, then $B_R$ is a betweenness relation.
\end{corollary}

Among the betweenness axioms, \eqref{BT0} and \eqref{BT1} have modal correspondents which is proved below in Theorem~\ref{th:full-complex-is-BTA}. However, in general, we have
\begin{theorem}[{\citealp{Duntsch-Orlowska-BS}}]\label{thm:BnotP}
The class of weak betweenness relations is not modal axiomatic.
\end{theorem}
\begin{proof}
    We are going to show that the class of weak betweenness relations is not closed under bounded morphisms. To this end, consider the set $\Intgr$ of all integers with the betweenness relation $B_{\leqslant}$ induced by the standard linear order $\leqslant$ on $\Intgr$. By Corollary~\ref{prop:B-from-poset} we have that $B_{\Intgr}$ is a~betweenness relation, and so, $\frZ\defeq\langle\Intgr,B_{\leqslant}\rangle$ is a b-frame. On the other hand, $\frF\defeq\langle\{w_0,w_1\},R\rangle$ where $w_0\neq w_1$ and $R\defeq\{w_0,w_1\}^3$ is not even a weak b-frame, as $R$ does not satisfy \eqref{BTW}.

    Let $f\colon\Intgr\to\{w_0,w_1\}$ be such that:
    \[
    f(x)\defeq\begin{cases}
      w_0,\quad&\text{if $x$ is even,}\\
      w_1,\quad&\text{if $x$ is odd.}
    \end{cases}
    \]
    Since $R$ is the universal ternary relation on $\{w_0,w_1\}$, $f$ preserves $B_{\leqslant}$, so to show that $f$ is indeed a bounded morphism all that is left is to prove that $f$ satisfies the back condition:
    \[
    R(u,f(x),v)\rarrow(\exists y,z\in\Intgr)\,(B_{\leqslant}(y,x,z)\wedge f(y)=u\wedge f(z)=v)\,.
    \]
    The proof will be done by cases:  Suppose $R(u,f(x),v)$. In case $f(x)=w_0$, we have that $x$ is even, and there are the following possibilities:
    \begin{enumerate}[label=(\arabic*),itemsep=0pt]
        \item $u=v=w_0$: Set $y\defeq x\eqdef z$. We have $B_{\leqslant}(x,x,x)$ and $f(y)=w_0$ and $f(z)=w_0$.
        \item $u=v=w_1$: Set $y\defeq x-1$ and $z\defeq x+1$. Thus $B_{\leqslant}(x-1,x,x+1)$ and $f(x-1)=w_1$ and $f(x+1)=w_1$.
        \item $u=w_0$ and $v=w_1$: Set $y\defeq x$ and $z\defeq x+1$. Thus $B_{\leqslant}(x,x,x+1)$ and $f(x)=w_0$ and $f(x+1)=w_1$.
        \item $u=w_1$ and $v=w_0$: Set $y\defeq x-1$ and $z\defeq x$. We have that $B_{\leqslant}(x-1,x,x)$ and $f(y)= w_1$ and $f(z)=w_0$.
    \end{enumerate}
    The proof for the case $f(x)=w_1$ is analogous.
\end{proof}

\begin{theorem}
    The class of betweenness relations is not sufficiency axiomatic.
\end{theorem}
\begin{proof}
    In light of Theorem~\ref{th:key-for-insufficiency} it is enough to show that the class of complementary 3-frames for betweenness frames is not possibility definable. This is the same as showing that the class of betweenness frames is not closed under co-bounded morphisms. Thus, we exhibit two 3-frames $\langle U,R\rangle$ and $\langle V,S\rangle$ as well as a co-bounded surjective morphism $p\colon U\to V$ such that $\langle U,R\rangle$ is a betweenness frame and $\langle V,S\rangle$ is not.

    Let $U\defeq \{a,b\}$ and $R\defeq\{\langle a,a,a\rangle,\langle b,b,b\rangle\}$; then $R$ is a betweenness relation. Furthermore, set $V\defeq\{x\}$ and $S\defeq\emptyset$; then $-S=\{\langle x,x,x\rangle\}$ and $S$ is not a betweenness relation since it does not satisfy \eqref{BT0}.

    Obviously, there is a unique surjection $p\colon U\to V$ given by $p(a)\defeq x\eqdef p(b)$, and it is a bounded morphism between $\langle U,-R\rangle$ and $\langle V,-S\rangle$. It preserves $-R$ as it is constant, and it satisfies the back condition since if e.g., $-S(x,p(a),x)$, then $-R(b,a,b)$ and $f(b)=x$.
\end{proof}

The proof reflects the fact that reflexivity of a binary relation is not definable by a unary sufficiency operator.

\section{Complex algebras of b-frames}\label{sec:cmcf}

To motivate the choice of axioms for abstract algebras of betweenness we will focus on b-frames and their complex algebras. We are going to show that axioms of b-frames  correspond to algebraic properties of their complex algebras, and the latter will serve in the next section as <<role models>> for the axioms expressed within the framework of PS-algebras.

For a $3$-frame $\frF\defeq\klam{U,B}$ we define its complex operators by
\begin{align}
\poss{B}(X,Y)\defeq{}&\{u\in U\mid(\exists x\in X)(\exists y\in Y)\,B(x,u,y)\}\tag{$\dftt{\poss{B}}$}\\
={}& \set{u \in U\mid (X \times \set{u} \times Y) \cap B \neq \z} \notag \\
\suff{B}(X,Y)\defeq{}&\{u\in U\mid(\forall x\in X)(\forall y\in Y)\,B(x,u,y)\}\tag{$\dftt{\suff{B}}$}\\
={}& \set{u \in U\mid X \times \set{u} \times Y \subseteq B}. \nonumber
\end{align}
Thus, $\Cm^{ps}(\frF)=\langle 2^U,\poss{B},\suff{B}\rangle$ is the full complex algebra of $\frF$. We will prove that the following conditions for complex algebras correspond to relational axioms for betweenness:
\begin{gather*}
    X\subseteq\poss{B}(X,X)\tag{\ref{BT0}$^\mathfrak{c}$}\,,\label{BT0c}\\
    \poss{B}(X,Y)\subseteq\poss{B}(Y,X)\tag{\ref{BT1}$^\mathfrak{c}_f$}\,,\label{BT1c}\\
    \suff{B}(X,Y)\subseteq\suff{B}(Y,X)\,,\tag{\ref{BT1}$^\mathfrak{c}_g$}\label{BT1cg}\\
    Y\cap\poss{B}(X,Z)\subseteq\poss{B}(X\cap\poss{B}(X,Y),Z)\tag{\ref{BT2}$^\mathfrak{c}$}\,,\label{BT2c}\\
    \poss{B}(X,\suff{B}(X,-Y)\cap Y)\subseteq Y\,,\tag{\ref{BT3}$^\mathfrak{c}$}\label{BT3c}\\
       \suff{B}(X,X)\subseteq X\,,\ \text{for all $X\neq\emptyset$}\,,\tag{\ref{BTW}$^\mathfrak{c}$}\label{BTWc} \\
     \tag{\ref{BT2s}$^\mathfrak{c}$}\label{BT2sc}
    X\subseteq\poss{B}(X,Y)\,,\ \text{for all}\ Y\neq\emptyset\,.
\end{gather*}
Note that all of these are universal, so they hold in all subalgebras of $\Cm^{ps}(\frF)$.
\begin{theorem}\label{th:BTi-iff-BTic}
    Let $\frF\defeq\langle U,B\rangle$ be a 3-frame. Then,
    $\frF$ satisfies \textup{(BT$i$)} \tiff $\Cm^{ps}(\frF)$ satisfies \textup{(BT$i^\mathfrak{c}$)}, for any $i\in\{0,1_f,1_g,2,3,\mathrm{W},2\mathrm{s}\}$.
\end{theorem}
\begin{proof}
($i=0$) Let $\frF$ satisfy \eqref{BT0} and take $x\in X$. Since $B(x,x,x)$, it is the case that $(X\times\{x\}\times X)\cap B\neq\emptyset$in consequence as required.

Conversely, given $x\in U$, we have $\{x\}\subseteq\poss{B}(\{x\},\{x\})$, which means that $B(x,x,x)$.

\smallskip

($i=1_f$) Suppose $\frF$ meets \eqref{BT1}. If $x\in\poss{B}(X,Y)$, then $(X\times\{x\}\times Y)\cap B\neq\emptyset$, and from \eqref{BT1} we obtain $(Y\times\{x\}\times X)\cap B\neq\emptyset$, i.e., $\poss{B}(Y,X)$.

For the reverse implication, suppose that $\frA$ satisfies \eqref{BT1c}. If $B(x,y,z)$, then $(\{x\}\times\{y\}\times\{z\})\cap B\neq\emptyset$, i.e., $y\in\poss{B}(\{x\},\{z\})$. By \eqref{BT1c} $\poss{B}(\{x\},\{z\})\subseteq\poss{B}(\{z\},\{x\})$, so $(\{z\}\times\{y\}\times\{x\})\cap B\neq\emptyset$, and so \eqref{BT1} holds for $\frF$.

\smallskip

($i=1_g$) If $z \in \suff{B}(X,Y)$, then $X\times\{z\}\times Y\subseteq B$. \eqref{BT1} implies that $Y\times\{z\}\times X\subseteq B$, and thus $z\in \suff{B}(Y,X)$.

\smallskip

Assume now that $\suff{B}(X,Y) \subseteq \suff{B}(Y,X)$ for all $X,Y \subseteq U$. If $B(x,z,y)$, then $z\in\suff{B}(\{x\},\{y\})$, and in consequence $z\in\suff{B}(\{y\},\{x\})$. Therefore $B(y,z,x)$.

($i=2$) Assume \eqref{BT2}, and let $y\in Y\cap\poss{B}(X,Z)$. This means that there are $x\in X$ and $z\in Z$ such that $B(x,y,z)$. By the axiom, $B(x,x,y)$, so $x\in\poss{B}(X,Y)$. Therefore,
\[\left[(X\cap\poss{B}(X,Y))\times\{y\}\times Z\right]\cap B\neq\emptyset\]
and in consequence we have
\[
y\in\poss{B}(X\cap\poss{B}(X,Y),Z)
\]
as required.

We now assume the inclusion $Y\cap\poss{B}(X,Z)\subseteq\poss{B}(X\cap\poss{B}(X,Y),Z)$ and that $B(x,y,z)$. Thus $y\in\{y\}\cap\poss{B}(\{x\},\{z\})$. Thus, it must be the case that:
\[
y\in\poss{B}(\{x\}\cap\poss{B}(\{x\},\{y\}),\{z\})
\]
which in particular means that
\[
\Bigl(\bigl[\{x\}\cap\poss{B}(\{x\},\{y\})\bigr]\times\{y\}\times\{z\}\Bigr)\cap B\neq\emptyset\,.
\]
On the other hand, this entails that $\{x\}\cap\poss{B}(\{x\},\{y\})\neq\emptyset$, and thus $B(x,x,y)$.\footnote{We would like to thank S{\o}ren Brinck Knudstorp, whose help was crucial to discover this inclusion.}

\smallskip

($i=3$) Assume \eqref{BT3} for $\frF$. Let $z\in\poss{B}(X,\suff{B}(X,-Y)\cap Y)$. Then, there are $x\in X$ and $a\in \suff{B}(X,-Y)\cap Y$ such that $B(x,z,a)$. Since $a$ is in $\suff{B}(X,-Y)$, for all $x'\in X$ and $b\notin Y$ it is the case that $B(x',a,b)$. If $z\notin Y$, then $B(x,a,z)$ and so \eqref{BT3} entails that $z=a$ and $z\in Y$. Therefore $z\in Y$ as required.

Let $\frA$ satisfy \eqref{BT3c}. Assume $B(x,y,a)$, $B(x,a,y)$ and $a\neq y$. Take $X\defeq\{x\}$ and $Y\defeq-\{a\}$. Since $\{x\}\times\{y\}\times\{a\}\subseteq B$ we have that $y\in\suff{B}(\{x\},\{a\})\cap-\{a\}$. Thus,
\[
a\in\poss{B}\left(\{x\},\suff{B}(\{x\},\{a\})\cap-\{a\}\right)\,.
\]
This yields a contradiction, as \eqref{BT3c} entails that $a\in-\{a\}$.

\smallskip

($i=\mathrm{W}$) Let $\frF$ satisfy \eqref{BTW}. If $y\in\suff{B}(X,X)$, then $X\times\{y\}\times X\subseteq B$. But there is $x\in X$, so $B(x,y,x)$ and by the axiom assumed, we have that $x=y$, i.e., $y\in X$.

If \eqref{BTWc} holds for $\frA$ and $B(x,y,x)$, then $\{x\}\times\{y\}\times\{x\}\subseteq B$, so $y\in\suff{B}(\{x\},\{x\})$ and in consequence $y\in\{x\}$, as required.

\smallskip

($i=2\mathrm{s}$)   Assume $B(a,a,b)$. Let $x\in X$. Since there is $y\in Y$ and $B(x,x,y)$ by the assumption, we have that $x\in\poss{B}(X,Y)$, as required.

For the reverse implication, $\{x\}\subseteq\poss{B}(\{x\},\{y\})$  for any $x$ and $y$, therefore, we have that $B(x,x,y)$.
\end{proof}

Complex algebras satisfy a property which connects $\poss{B}$ and $\suff{B}$ and which will become important in the next section:

\begin{lemma}\label{lem:g2f2Cm}
If $ X,Y \neq \z$, then
\begin{equation}
\suff{B}(X,Y) \subseteq \poss{B}(X,Y)
\tag{\ref{g2f2}$^\mathfrak{c}$}\,. \label{dagc}
\end{equation}
\end{lemma}
\begin{proof}
Suppose that $X, Y \neq \z$,  and that $u \in \suff{B}(X,Y)$; then, $X \times \set{u} \times Y \subseteq B$ by definition of $\suff{B}$. Since $X,Y \neq \z$, there are  $x \in X$, $y \in Y$ with $B(x,u,y)$. This shows that $u \in \poss{B}(X,Y)$.
\end{proof}




Let us observe that the betweenness relation of a b-frame $\frF$ can be characterized by means of the sufficiency operator of the full complex algebra of $\frF$:
\begin{proposition}
    If $\frF\defeq\langle U,B\rangle$ is a betweenness frame and $\Cm^{ps}(\frF)$ is its full complex betweenness algebra, then\/\textup{:}
    \[
        B=\bigcup_{\langle X,Y\rangle\in 2^U\times 2^U}X\times\suff{B}(X,Y)\times Y\,.
    \]
\end{proposition}
\begin{proof}
($\subseteq$) If $B(x,y,z)$, then $y\in\suff{B}(\{x\},\{z\})$, hence,
\[
\langle x,y,z\rangle\in \{x\}\times \suff{B}(\{x\},\{z\})\times\{z\}\,.
\]

\smallskip

($\supseteq$) Now assume that $X$ and $Y$ are subsets of $U$ such that
\[
\langle x,z,y\rangle\in X\times\suff{B}(X,Y)\times Y\,.
\]
Therefore, $x\in X$, $y\in Y$ and $z\in\suff{B}(X,Y)$. From the last condition we obtain
\[
X\times\{z\}\times Y\subseteq B\,,
\]
which implies $B(x,z,y)$.
\end{proof}

\section{Betweenness algebras}\label{sec:BetAlg}

Theorem~\ref{th:BTi-iff-BTic} motivates our choice of axioms towards an abstract algebraization of betweenness, i.e., we translate the counterparts of the betweenness axioms for 3-frames in the <<obvious>> way:

\begin{definition}
    A PS-algebra $\frA\defeq\langle A,f,g\rangle$ is a \emph{betweenness algebra} (\emph{b-algebra} for short) if $\frA$ satisfies the following axioms:
    \begin{gather}
        x\leq f(x,x)\,,\tag{ABT0}\label{ABT0}\\
        f(x,y)\leq f(y,x)\,,\tag{ABT1$_f$}\label{ABT1}\\
         g(x,y) \leq  g(y,x)\,,\tag{ABT1$_g$}\label{ABT1g}\\
        y\cdot f(x,z)\leq f(x\cdot f(x,y),z)\,,\tag{ABT2}\label{ABT2}\\
        f(x,g(x,-y)\cdot y)\leq y\,,\tag{ABT3}\label{ABT3}
     \end{gather}
    and
    \begin{equation}
    x \neq \zero \tand y \neq \zero \rarrow g(x,y) \leq f(x,y). \tag{wMIA}\label{g2f2}
    \end{equation}\QED
    \end{definition}
%

\pagebreak

It can be shown that \eqref{g2f2} is independent of the other axioms. Moreover, let us observe that \eqref{ABT0} is equivalent to
\[
x\cdot y\leq f(x,y).
\]
\begin{proof}
    From \eqref{ABT0} we obtain $x\cdot y\leq f(x\cdot y,x\cdot y)$, and so monotonicity of $f$ entails that $x\cdot y\leq f(x,y)$. 
\end{proof}

The following example shows that \eqref{ABT1}  and \eqref{ABT1g} are independent of each other:

\begin{example}
Consider $\frA_1 \defeq \klam{A,f_1,g_1}$ and $\frA_2\defeq \klam{A,f_2,g_2}$ where $A$ is the four element Boolean algebra with atoms $a,b$, and $f_1, f_2$ and $g_1,g_2$ are given by the tables below:
\[
 \begin{array}{c|cccccc|cccc}
f_1 & \zero & a & b & \one &&g_1& \zero & a & b &\one    \\ \hline
\zero & \zero & \zero &\zero &\zero  && \zero & \one & \one & \one & \one \\
a & \zero & \one & \one & \one && a & \one & \zero & b & \zero\\
b &\zero &\one & \one &\one && b &\one &\zero &b &\zero\\
\one & \zero &\one &\one &\one && \one &\one &\zero &b &\zero
 \end{array}\qquad\quad
 \begin{array}{c|cccccc|cccc}
f_2 & \zero & a & b & \one & &g_2& \zero & a & b &\one    \\ \hline
\zero & \zero & \zero &\zero &\zero  & & \zero & \one & \one & \one & \one \\
a & \zero & a & a & a &  & a & \one & \zero & \zero & \zero\\
b &\zero &\one & \one &\one &  & b &\one &\zero &\zero &\zero\\
\one & \zero &\one &\one &\one &  & \one &\one &\zero &\zero &\zero
 \end{array}
\]
We can see that $f_1(x,y)=f_1(y,x)$  for all $x$ and $y$, but $g_1(a,b) \neq g_1(b,a)$ in $\frA_1$, and $g_2(x,y) = g_2(y,x)$ but $f_2(a,b) \neq f_2(b,a)$ in $\frA_2$. This example justifies the inclusion of both \eqref{ABT1} and \eqref{ABT1g} in the set of postulates for betweenness algebras. Both axioms are counterparts of the symmetry axiom \eqref{BT1}. \QED
\end{example}

\begin{definition}
    If $\frA$ satisfies \eqref{ABT0}--\eqref{ABT2} and
    \begin{equation}
        a\neq\zero\rarrow g(a,a)\leq a\,,\tag{ABTW}\label{ABTW}
    \end{equation}
    then $\frA$ is a \emph{weak betweenness algebra}. If $\frA$ satisfies \eqref{ABT1}, \eqref{ABT3}, \eqref{g2f2} and
    \begin{equation}\tag{ABT2$^\mathrm{s}$}\label{ABT2s}
    b\neq\zero\rarrow a\leq f(a,b)\,,
    \end{equation}
    then it will be called a \emph{strong betweenness algebra}.\QED
\end{definition}

\begin{proposition}\label{prop:b-algebra-is-weak-b}
    Every betweenness algebra is a weak betweenness algebra.
\end{proposition}
\begin{proof}
    Let $x\neq \zero$; we will show that $g(x,x)\leq x$. It follows from \eqref{ABT3} that $f(x,g(x,x)\cdot -x)\leq -x$, and by \eqref{ABT2} we have
    \[
    -x\cdot g(x,x)\cdot f(x,x)\leq f\bigl(x\cdot f(x,-x\cdot g(x,x)),x\bigr)\leq f(x\cdot -x,x)=\zero\,.
    \]
    Thus, $-x\cdot g(x,x)\cdot f(x,x)=\zero$. Since $x\neq\zero$ by the assumption, we apply \eqref{g2f2} to obtain that $g(x,x)\leq f(x,x)$. Thus $-x\cdot g(x,x)=\zero$, i.e.,  $g(x,x)\leq x$.
\end{proof}

\begin{proposition}
    \eqref{ABT2s} entails both \eqref{ABT0} and \eqref{ABT2}, therefore, every strong betweenness algebra is a betweenness algebra.
\end{proposition}
\begin{proof}
\eqref{ABT0}: If $x=\zero$, then immediately $x\leq f(x,x)$. If $x\neq\zero$, then from the assumption we obtain $x\leq f(x,x)$.

\smallskip

\eqref{ABT2}: Fix arbitrary $x, y$ and $z$. We have two possibilities, either $y=\zero$ or $y\neq\zero$. In the former, we have $\zero=y\cdot f(x,z)\leq f(x\cdot f(x,y),z)$. In the latter, directly from \eqref{ABT2s} we obtain $x=x\cdot f(x,y)$ so $y\cdot f(x,z)\leq f(x,z)= f(x\cdot f(x,y),z)$.
\end{proof}

From Lemma~\ref{lem:g2f2Cm} and Theorem~\ref{th:BTi-iff-BTic} we see that the axioms hold in complex algebras of b-frames:
\begin{theorem}\label{th:full-complex-is-BTA}
    If $\frF$ is a (weak, strong) betweenness frame, then $\Cm^{ps}(\frF)$ is a (weak, strong) betweenness algebra.
\end{theorem}

The definition of sufficiency operator implies that $g(\zero,a) = g(a,\zero) = \one$ for all $a \in A$. Our next result shows under which other conditions $g(a,b) = \one$ is possible:

\begin{theorem}\label{th:g-1-atoms}
Suppose that $\frA \df \klam{A,f,g}$ is a b-algebra, $\zero \not\in \set{a,b,c} \subseteq A$ and $g(a,b) = \one$. Then,
\begin{enumerate}
\item $\card{A} = 2$ or $a \cdot b = \zero$.
\item $a,b \in \At(A)$.
\item If $g(a,c) = \one$, then $b = c$.
\end{enumerate}
\end{theorem}
\begin{proof}
1. Suppose that $a \cdot b \neq \zero$. Then, by \eqref{ABTW} and the fact that $g$ is antitone,
\begin{gather*}
\one = g(a,b) \leq g(a \cdot b, a\cdot b) \leq a \cdot b,
\end{gather*}
which implies $a = b =\one$. If $0 \lneq c \leq\one$, then $\one= g(\one,\one) \leq g(c,c) \leq c$ which shows that $\card{A} = 2$.

2. If $\card{A} = 2$, then $a = b = \one \in \At(A)$. Thus, $a \cdot b = \zero$ by 1. above. Suppose that $0 \leq c \lneq b$; then $0 \neq b\cdot-c\neq\zero$, hence, $\one = g(a,b)\leq g(a,b\cdot-c)\leq f(a,b\cdot-c)$ by \eqref{g2f2}. Furthermore, $\one = g(a,b) \leq g(a,c)$, therefore, $b\cdot-c\leq -c =  g(a,c)\cdot-c$. Using \eqref{ABT3} and the fact that $f$ is monotone we obtain
\begin{gather*}
\one=f(a,b\cdot-c)\leq f(a,g(a,c)\cdot-c))\leq-c,
\end{gather*}
which implies $c = \zero$. It is shown analogously that $a$ is an atom.

3. If $g(a,b) = g(a,c) = \one$, then both $b$ and $c$ are atoms by 2. above and $g(a,b) \cdot g(a,c) = g(a, b+c) = \one$. Again by 2. above we obtain that $b+c$ is an atom, which implies $b = c$.
\end{proof}

\begin{corollary}
Suppose that $\frA \df \klam{A,f,g}$ is a b-algebra. Then,
\begin{enumerate}
  \item If $A$ has at most one atom, then $g^{-1}[\{\one\}]=(\{\zero\}\times A)\cup(A\times\{\zero\})$.
  \item If $M\subseteq (A\setminus\{\zero\})\times(A\setminus\{\zero\})$ and $g[M]=\{\one\}$, then $M\subseteq\At(A)^2$ and $M$ is an antichain in the product order of $A^2$.
\end{enumerate}
\end{corollary}

\begin{figure}
\centering
\begin{tikzpicture}
    \foreach \x\n in {-2/a,0/b,2/c}
        {
            \node [fill,circle,inner sep=2pt] (\n) at (\x,0) {};
        }
        \node [fill,circle,inner sep=2pt] (z) at (0,-1.5) {};
        \node [fill,circle,inner sep=2pt] (o) at (0,3) {};
        \node[below,outer sep = 2pt] at (z) {$\zero$};
        \node[above,outer sep = 2pt] at (o) {$\one$};
    \foreach \y\z in {a/-2,b/0,c/2}
    {
        \draw (z) -- (\y);
        \draw[dotted,shorten <=5pt,shorten >=5pt] (\y) -- (\z,1);
    }
    \draw[dotted,shorten <=5pt,shorten >=5pt] (a) -- (b);
    \draw[dotted,shorten <=5pt,shorten >=5pt] (-3,0) -- (a);
    \draw[dotted,shorten <=5pt,shorten >=5pt] (b) -- (c);
    \draw[dotted,shorten <=5pt,shorten >=5pt] (c) -- (3,0);
    \draw[dotted,shorten <=5pt,shorten >=5pt] (o) -- (0,2);

    \coordinate (c1) at (-1,0);
    \coordinate (c2) at (1,0);

    \node[ellipse,draw,minimum width = 3cm, minimum height = 1.2cm,gray] (e1) at (c1) {};
    \node[ellipse,draw,minimum width = 3cm, minimum height = 1.2cm,gray] (e2) at (c2) {};
    \draw[->,shorten <=5pt,shorten >=5pt,teal] (e1) to [out=90,in=180] node [pos = 0.5,left = 3pt] {$g$} (o);
    \draw[->,shorten <=5pt,shorten >=5pt,teal] (e2) to [out=90,in=0] node [pos = 0.5,right = 3pt] {$g$} (o);
\end{tikzpicture}
\caption{The situation excluded by Theorem~\ref{th:g-1-atoms}(3).}\label{fig:impossible}
\end{figure}
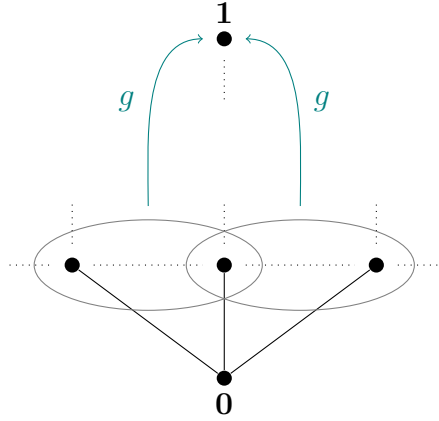

\begin{example}\label{ex:g-on-pairs-of-atoms}
With respect to Theorem~\ref{th:g-1-atoms}(3) we will exhibit a b-algebra $\frA$ such that $g(a,b)=\one=g(c,d)$, none of the four arguments is zero and they are pairwise disjoint, i.e., the theorem is as strong as it can be. To this end, let $U\defeq\{a,b,c,d\}$, take all triples
\[
M\defeq\{\klam{a,x,b}\mid x\in U\}\cup\{\klam{c,x,d}\mid x\in U\},
\]
and close $M$ under axioms \eqref{BT0}--\eqref{BT3} to obtain $B$. Such closure does not lead to a contradiction, that is, no pair of triples $\klam{x,y,z}$ and $\klam{x,z,y}$ such that $y\neq z$ is contained in $B$. In consequence, $\frF\defeq\klam{U,B}$ is a b-frame. In $\Cm^{ps}(\frA)$ we have
$\suff{B}(\{a\},\{b\})=U=\suff{B}(\{c\},\{d\})$. The situation is depicted in Figure~\ref{fig:g-on-pairs-of-atoms}.\QED
\end{example}

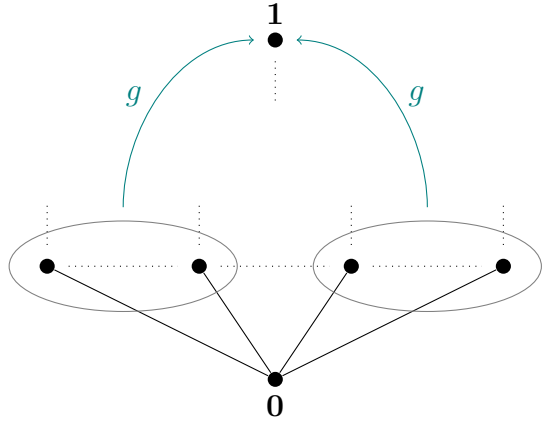
\begin{figure}[htb]
\centering
\begin{tikzpicture}
    \foreach \x\n in {-3/a,-1/b,1/c,3/d}
        {
            \node [fill,circle,inner sep=2pt] (\n) at (\x,0) {};
        }
        \node [fill,circle,inner sep=2pt] (z) at (0,-1.5) {};
        \node [fill,circle,inner sep=2pt] (o) at (0,3) {};
        \node[below,outer sep = 2pt] at (z) {$\zero$};
        \node[above,outer sep = 2pt] at (o) {$\one$};
    \foreach \y\z in {a/-3,b/-1,c/1,d/3}
    {
        \draw (z) -- (\y);
        \draw[dotted,shorten <=5pt,shorten >=5pt] (\y) -- (\z,1);
    }
    \draw[dotted,shorten <=5pt,shorten >=5pt] (a) -- (b);
    \draw[dotted,shorten <=5pt,shorten >=5pt] (b) -- (c);
    \draw[dotted,shorten <=5pt,shorten >=5pt] (c) -- (d);
    \draw[dotted,shorten <=5pt,shorten >=5pt] (o) -- (0,2);

    \coordinate (c1) at (-2,0);
    \coordinate (c2) at (2,0);

    \node[ellipse,draw,minimum width = 3cm, minimum height = 1.2cm,gray] (e1) at (c1) {};
    \node[ellipse,draw,minimum width = 3cm, minimum height = 1.2cm,gray] (e2) at (c2) {};
    \draw[->,shorten <=5pt,shorten >=5pt,teal] (e1) to [out=90,in=180] node [pos = 0.5,left = 3pt] {$g$} (o);
    \draw[->,shorten <=5pt,shorten >=5pt,teal] (e2) to [out=90,in=0] node [pos = 0.5,right = 3pt] {$g$} (o);
\end{tikzpicture}\caption{As Example~\ref{ex:g-on-pairs-of-atoms} shows, there may be disjoint pairs of atoms for which $g$ takes the value~$\one$.}\label{fig:g-on-pairs-of-atoms}
\end{figure}

\begin{theorem}\label{thm:ABTW+}
Suppose that $\klam{A,g}$ is a binary sufficiency algebra which satisfies \eqref{ABTW}.
 If $g(a,b) \neq \zero$, then $a \cdot b \leq g(a,b)$. If, additionally, $a\cdot b\neq \zero$, then $g(a,b) = a \cdot b$ and $a \cdot b$ is an atom. 
\end{theorem}
\begin{proof}
Assume that $a \cdot b \not\leq g(a,b)$. Then, ${a \cdot b \cdot -g(a,b) \neq \zero}$, and \eqref{ABTW} implies that
\begin{gather*}
g(a \cdot b \cdot  -g(a,b),a \cdot -g(a,b)) \leq a \cdot b \cdot -g(a,b).
\end{gather*}
 Since $g$ is antitone we obtain
\[
g(a,b) \leq g(a \cdot b \cdot -g(a,b),a \cdot b \cdot -g(a,b)) \leq a \cdot b \cdot -g(a,b) \leq -g(a,b),
\]
and thus, $g(a,b) =\zero$. If, in addition,  $a \cdot b \neq \zero$, then $g(a \cdot b, a \cdot b) \leq a \cdot b $ by \eqref{ABTW}, thus, $g(a,b) = a \cdot b$.

Finally, suppose that $\zero \lneq c \leq a \cdot b$. Using \eqref{ABTW} and the fact that $g$ is antitone we obtain
\begin{gather*}
\zero \neq c \leq a \cdot b \leq g(a,b) \leq g(a \cdot b, a \cdot b) \leq g(c,c) \leq c.
\end{gather*}
It follows that $c = a \cdot b$, thus, $a \cdot b$ is an atom of $A$.
%
\end{proof}

\begin{example}
    Consider the b-frame $\frF\defeq\langle[0,1],\Bleqs\rangle$ where $[0,1]$ is the closed unit interval of the real numbers and $\Bleqs$ is induced by the standard order $\leqslant$.  In $\Cm^{ps}(\frF)$ we have $\suff{\Bleqs}(\{0\},\{1\})=[0,1]$, and clearly $\{0\}$ and $\{1\}$ are disjoint atoms of the complex algebra of $\frF$. It is easy to see that if $\suff{\Bleqs}(X,Y)=[0,1]$, then $\{X,Y\}=\{\{0\},\{1\}\}$.

    On the other hand, for the subframe $\frF^\ast\defeq\klam{[0,1),\Bleqs^\ast}$ of $\frF$, we have that:
    \[
        \suff{\Bleqs^\ast}^{-1}\left[\{[0,1)\}\right]=\left(\{\emptyset\}\times 2^{[0,1)}\right)\cup\left(2^{[0,1)}\times \{\emptyset\}\right),
    \]
    although $\frF^\ast$ has uncountably many atoms.\QED
\end{example}

\begin{lemma}\label{lem:for-d}
Let $\frA\defeq\klam{A,f,g}$ be a PS-algebra. In the presence of  \eqref{ABT0}, \eqref{ABT2} and \eqref{ABT3} the axiom  \eqref{ABTW} is equivalent to
\begin{gather}\label{5-for-d}
(\forall a,b \in A)[a \cdot b \neq \zero \rarrow g(a,b) \leq f(a,b)].
\end{gather}
\end{lemma}
\begin{proof}
    ($\rarrow$) Suppose $a\cdot b\neq\zero$. Then, using \eqref{ABTW} and \eqref{ABT0} we have:
    \[
        g(a,b)\leq g(a\cdot b,a\cdot b)\leq a\cdot b\leq f(a,b)\,.
   \]
($\larrow$) Suppose $a\neq\zero$. By \eqref{ABT3} it is the case that $f(a,g(a,a)\cdot-a)\leq-a$, and \eqref{ABT2} entails that
\[
-a\cdot g(a,a)\cdot f(a,a)\leq f(a\cdot f(a,g(a,a)\cdot-a),a)\leq f(\zero,a)=\zero\,.
\]
Since \eqref{5-for-d} entails that $g(a,a)\leq f(a,a)$, we have $-a\cdot g(a,a)=\zero$, as required.
\end{proof}

Clearly, \eqref{g2f2} implies \eqref{5-for-d}, but the converse does not hold as the algebra with atoms $a,b,c$ of Table \ref{tab:abtwd} shows.
\begin{table}[htb]  \centering
\[
\begin{array}{r|rrrrrrrr}
f & \zero & \one & b &ac & c & ab & bc & a\\
\hline
    \zero & \zero & \zero & \zero & \zero & \zero & \zero & \zero & \zero \\
    \one & \zero & \one & \one & \one & \one & \one & \one & \one \\
    b & \zero & \one & \one & \one & bc & \one & \one & ab \\
   ac & \zero & \one & \one & \one & \one & \one & \one & \one \\
    c & \zero & \one & bc & \one & \one & \one & \one &ac \\
    ab & \zero & \one & \one & \one & \one & \one & \one & \one \\
    bc & \zero & \one & \one & \one & \one & \one & \one & \one \\
    a & \zero & \one & ab & \one &ac & \one & \one & \one
\end{array}
\qquad
\begin{array}{r|rrrrrrrr}
g & \zero & \one & b &ac & c & ab & bc & a\\
\hline
    \zero & \one & \one & \one & \one & \one & \one & \one & \one \\
    \one & \one & \zero & \zero & \zero & \zero & \zero & \zero & \zero \\
    b & \one & \zero & \zero & \zero & a & \zero & \zero & \zero \\
   ac & \one & \zero & \zero & \zero & \zero & \zero & \zero & \zero \\
    c & \one & \zero & a & \zero & \zero & \zero & \zero & \zero \\
    ab & \one & \zero & \zero & \zero & \zero & \zero & \zero & \zero \\
    bc & \one & \zero & \zero & \zero & \zero & \zero & \zero & \zero \\
    a & \one & \zero & \zero & \zero & \zero & \zero & \zero & \zero
\end{array}
\]\caption{A b-algebra satisfying strong b-axioms, \eqref{5-for-d}, and not \eqref{g2f2}. For simplicty, $xy$ is $x+y$.}\label{tab:abtwd}
\end{table}

Algebras which satisfy \eqref{g2f2} are binary instances of the weak MIAs introduced in Section~\ref{sec:bao}. A remarkable property of such algebras is the fact that they are discriminator algebras in the sense of \cite{wer78}. Here, the \emph{unary discriminator} on a Boolean algebra $A$ is the mapping $d\colon A \to A$ such that for all $a \in A$
\begin{gather*}
d(a) =
\begin{cases}
\zero, &\text{if } a = \zero, \\
\one, &\text{otherwise.}
\end{cases}
\end{gather*}

\pagebreak

\noindent However, as the lemma below demonstrates, \eqref{g2f2} can be weakened to \eqref{5-for-d}.

\begin{lemma}
Suppose that $\frA \defeq \klam{A,f,g}$ is a binary PS-algebra. Then $d: A \to A$ as defined above is the unary discriminator \tiff
$\frA$ satisfies \eqref{5-for-d}.
\end{lemma}
\begin{proof}
\aright Suppose that $d$ is the discriminator, and let $a,b \in A, a \cdot b \neq 0$. Then, $d(a \cdot b) = f(a \cdot b, a \cdot b) + - g(a \cdot b, a \cdot b) = \one$, i.e. $g(a \cdot b, a \cdot b) \leq f(a \cdot b, a \cdot b)$. Now,
\begin{gather*}
g(a,b) \leq g(a \cdot b, a \cdot b) \leq  f(a \cdot b, a \cdot b) \leq f(a,b),
\end{gather*}
since $g$ is antitone and $f$ is monotone.

\aleft Let $a \in A, a \neq \zero$. Then, $g(a,a) \leq f(a,a)$ by the hypothesis, which implies that $f(a,a) + -g(a,a) = \one$.
\end{proof}

Furthermore, observe that ---due to the facts that $g$ is antitone and $f$ is monotone ---\eqref{5-for-d} is equivalent to
\[a \neq \zero  \rarrow g(a,a) \leq f(a,a)\,.\]

Collecting all the facts above we have the following
\begin{theorem}
Suppose that $\frA\defeq\klam{A,f,g}$ is a PS-algebra that satisfies \eqref{ABT0}, \eqref{ABT2} and \eqref{ABT3}. Then, $d\colon A \to A$  defined by $d(a) \defeq f(a,a)+ -g(a,a)$ is the unary discriminator \tiff $\frA$ satisfies \eqref{ABTW}.
\end{theorem}
\begin{proof}
\aright Suppose that $d$ is the discriminator, and let $\zero \neq a \in \frA$. Then, $d(a) = f(a,a) + -g(a,a) =\one$ which implies $g(a,a) \leq f(a,a)$. So \eqref{5-for-d} holds which entails \eqref{ABTW}, by Lemma~\ref{lem:for-d}.

\aleft Let $a \in A, a \neq \zero$. Then, $g(a,a) \leq f(a,a)$ by the hypothesis, which implies that $f(a,a) + -g(a,a) = \one$.
\end{proof}
\begin{corollary}\label{cor:wdisc}
Each weak b-algebra is a discriminator algebra.
\end{corollary}

Since a discriminator algebra is simple by \citep[Theorem 2.2]{wer78} this implies

\begin{proposition}
The class of weak b-algebras is not closed under products.
\end{proposition}

Since  every b-algebra is a weak b-algebra we obtain

\begin{corollary}\label{cor:disc}
Each b-algebra is a discriminator algebra.
\end{corollary}

\section{Representation of b-algebras}\label{sec:repAbtw}
In this section we shall show that the class of b-algebras is closed under canonical extensions.
Suppose that $\frA \defeq \klam{A,f,g}$ is a PS-algebra. Modifying the notation in Section \ref{sec:bao} for our purpose we designate the middle column as special and define canonical  relations on $\Ult(A)$ by
\begin{align}
Q_f(\ult_1,\ult_2,\ult_3) &\iffdef f[\ult_1 \times \ult_3] \subseteq \ult_2,\tag{$\dftt{Q_f}$} \\
S_g(\ult_1,\ult_2,\ult_3) &\iffdef g[\ult_1\times\ult_3]\cap\ult_2\neq\emptyset.\tag{$\dftt{S_g}$}
\end{align}
So, $\Cf^{ps}(\frA)=\langle\Ult(A),Q_f,S_g\rangle$ is the canonical frame of $\frA$.

\pagebreak

The weak MIA axiom \eqref{g2f2} shows that the canonical relations are not independent:

\begin{lemma}\label{lem:S2Q2}
$\frA$ satisfies \eqref{g2f2} \tiff $S_{g} \subseteq Q_{f}$. In particular,  $S_{\suff{B}}\subseteq Q_{\poss{B}}$ for every $\Cm^{ps}(\frF)$ of a b-frame $\frF$.
\end{lemma}
\begin{proof}
\aright Suppose that $S_{g}(\ult_1,\ult_2,\ult_3)$. Then, $\ult_2 \cap g[\ult_1 \times \ult_3] \neq \z$, say, $x \in \ult_1, y \in \ult_3$ and $g(x,y) \in \ult_2$. We need to show that $f[\ult_1 \times \ult_3] \subseteq \ult_2$. Let $s \in \ult_1,$ and $t \in \ult_3$; then, $\zero \neq x \cdot s \in \ult_1$ and $\zero \neq y \cdot t \in \ult_3$. By \eqref{g2f2} we have $g(x \cdot s,y \cdot t) \leq f(x \cdot s, y \cdot t)$. Using $g(x,y) \in \ult_2$ and the facts that $g$ is antitone and $f$ is monotone we obtain%
\begin{gather*}
 g(x,y) \le g(x \cdot s, y \cdot t)  \leq f(x \cdot s, y \cdot t) \leq f(s,t) \in \ult_2.
\end{gather*}
\aleft Let $x,y \neq \zero$ and assume that $g(x,y) \not\leq f(x,y)$, i.e., $g(x,y) \cdot -f(x,y) \neq \zero$. Choose ultrafilters $\ult_1,\ult_2,\ult_3$ such that $x \in \ult_1, y\in \ult_3$, and $g(x,y) \cdot -f(x,y) \in \ult_2$. Then, $g[\ult_1 \times \ult_3] \cap \ult_2 \neq \z$, and therefore, $f[\ult_1 \times \ult_3] \subseteq \ult_2$, in particular, $f(x,y) \in \ult_2$. This contradicts $-f(x,y) \in \ult_2$.
\end{proof}

The crucial observation about canonical frames of b-algebras is the following result:
\begin{lemma}\label{lem:cmp}
    If $\frA \defeq \klam{A,f,g}$ is a betweenness algebra, then its canonical frame $\Cf^{ps}(\frA)=\klam{\Ult(A),Q_f,S_g}$ has the following properties\/\textup{:}
    \begin{gather}
        Q_{f}(\ult,\ult,\ult)\,,\label{BT0CF}\tag{\ref{BT0}$^\mathsf{Cf}$}\\
        Q_{f}(\ult_1,\ult_2,\ult_3)\rarrow Q_{f}(\ult_3,\ult_2,\ult_1)\,,\label{BT1CF}\tag{\ref{BT1}$^\mathsf{Cf}_f$}\\
        S_g(\ult_1,\ult_2,\ult_3)\rarrow S_g(\ult_3,\ult_2,\ult_1)\,,\label{BT1CFg}\tag{\ref{BT1}$^\mathsf{Cf}_g$}\\
         Q_{f}(\ult_1,\ult_2,\ult_3)\rarrow Q_{f}(\ult_1,\ult_1,\ult_2)\,,\label{BT2CF}\tag{\ref{BT2}$^\mathsf{Cf}$}\\
        Q_{f}(\ult_1,\ult_2,\ult_3)\wedge S_{g}(\ult_1,\ult_3,\ult_2)\rarrow\ult_2=\ult_3\,.\label{BT3CF}\tag{\ref{BT3}$^\mathsf{Cf}$}
    \end{gather}
    If $\frA$ is a strong betweenness algebra, then $\Cf^{ps}(\frA)$ additionally satisfies
    \begin{equation}\label{BT2sCF}\tag{\ref{BT2s}$^\mathsf{Cf}$}
        Q_f(\ult_1,\ult_1,\ult_2)\,.
    \end{equation}
If $\frA$ is a weak betweenness algebra, then $\Cf^{ps}(\frA)$ meets
\begin{equation}\label{BTWCF}\tag{\ref{BTW}$^\mathsf{Cf}$}
    S_g(\ult_1,\ult_2,\ult_1)\rarrow\ult_1=\ult_2\,.
\end{equation}

\end{lemma}
\begin{proof}
    \eqref{BT0CF}: Fix an ultrafilter $\ult$. If $x,y\in\ult$, then, since $x\cdot y\leq f(x,y)$, it is the case that $f[\ult\times\ult]\subseteq\ult$, as every ultrafilter is closed under meets.

    \smallskip

    \eqref{BT1CF}: By \eqref{ABT1}.

    \smallskip

    \eqref{BT1CFg}: Let $S_g(\ult_1,\ult_2,\ult_3)$, i.e., $g[\ult_1\times\ult_3]\cap\ult_2\neq\emptyset$. Pick $x\in\ult_1$ and $y\in\ult_2$ such that $g(x,y)\in\ult_2$. By \eqref{ABT1g} it is the case that $g(y,x)\in\ult_2$, and, since $\langle y,x\rangle\in\ult_3\times\ult_1$, we have $S_g(\ult_3,\ult_2,\ult_1)$.

    \smallskip

    \eqref{BT2CF}: Suppose that $Q_{f}(\ult_1,\ult_2,\ult_3)$, i.e., $f[\ult_1 \times \ult_3] \subseteq \ult_2$.
Let $x \in \ult_1$ and $y \in \ult_2$; we need to show that $f(x,y) \in \ult_1$. Suppose, on the contrary, that $-f(x,y)\in \ult_1$. First, note that due to monotonicity of $f$ we have:
\[x\cdot -f(x,y)\cdot f(x\cdot -f(x,y),y)\leq x\cdot -f(x,y)\cdot f(x,y)=\zero.\]
By the assumption, $y\cdot f(x\cdot -f(x,y),\one)\in\ult _2$, but  \eqref{ABT2} implies
\[y\cdot f(x\cdot -f(x,y),\one)\leq f(x\cdot -f(x,y)\cdot f(x\cdot -f(x,y),y),\one)=f(\zero,\one)=\zero, \]
which is a contradiction. It follows that $f(x,y)\in \ult_1$.

\smallskip

 \eqref{BT3CF}: Assume $f[\ult_1\times\ult_3]\subseteq\ult_2$ and $g[\ult_1\times\ult_2]\cap\ult_3\neq\emptyset$. Let $x\in\ult_2$ and $-x\in\ult_3$. Furthermore, let $a\in\ult_1$ and $b\in\ult_2$ be such that $g(a,b)\in\ult_3$. Since $g$ is antitone we have that $g(a,b\cdot x)\in\ult_3$. Since also $-b+-x\in\ult_3$, it follows that $g(a,b\cdot x)\cdot(-b+-x)\in\ult_3$, and from \eqref{ABT3} we obtain
    \[
        \ult_2\ni f(a,g(a,b\cdot x)\cdot(-b+-x))\leq-b+-x\in\ult_2\,.
    \]
    But $b\cdot x\in\ult_2$, a contradiction. So it must be the case that $\ult_2=\ult_3$, as required.

    \smallskip

   \eqref{BT2sCF}: We need to show that $f[\ult_1\times\ult_2]\subseteq\ult_1$. If $x\in \ult_1$ and $y\in\ult_2$, then $y\neq\zero$ and by \eqref{ABT2s} we obtain that $x\leq f(x,y)$. Therefore $f(x,y)\in\ult_1$, as required.

   \smallskip

   \eqref{BTWCF}: Assume that $S_g(\ult_1,\ult_2,\ult_1)$, that is $g[\ult_1\times\ult_1]\cap\ult_2\neq\zero$. Let $y_1,y_2\in\ult_1$ such that $g(y_1,y_2)\in\ult_2$. So, from the fact that $g$ is antitone we obtain that $g(y_1\cdot y_2,y_1\cdot y_2)\in\ult_2$. If $x\in\ult_1$, then using the property one more time we obtain that $g(x\cdot y_1\cdot y_2,x\cdot y_1\cdot y_2)\in\ult_2$. However, $x\cdot y_1\cdot y_2\neq\zero$, therefore, from \eqref{ABTW} we obtain
   \[
g(x\cdot y_1\cdot y_2,x\cdot y_1\cdot y_2)\leq x\cdot y_1\cdot y_2\,,
   \]
   and so $x$ is in $\ult_2$. Thus, $\ult_1=\ult_2$ by maximality of ultrafilters.
 \end{proof}

We now have the following representation theorem:
\begin{theorem}\label{thm:repalg}
  Suppose that $\frA \defeq \klam{A,f,g}$ is a (weak, strong) betweenness algebra. Then, the Stone map $h\colon \frA \into  \Em^{ps}(\frA)$ is a (weak, strong) betweenness algebra embedding.
\end{theorem}
\begin{proof}
  $h\colon \frA \into\Em^{ps}(\frA)$ is a PS-algebra embedding by Theorems \ref{thm:repP} and \ref{thm:repS}, and we need to show that $\Em^{ps}(\frA)$ is a  (weak, strong) betweenness algebra; we will use Lemma \ref{lem:cmp}. Suppose that $\z \neq X,Y,Z \subseteq \Ult(A)$. Let us begin with the proof for a~betweenness algebra~$\frA$.

\smallskip

\eqref{ABT0}: Let $\ult \in X$. By \eqref{BT0CF} we have $Q_f(\ult,\ult,\ult)$, and $\ult\in X$ implies that  $(X \times \set{\ult} \times X) \cap Q_f \neq \z$. It follows that $\ult \in \poss{Q_f}(X,X)$.

\smallskip

\eqref{ABT1}: Let $\ult\in \poss{Q_f}(X,Y)$. Then, $(X \times \set{\ult} \times Y) \cap Q_f \neq \z$, say, $\ult_1 \in X, \ult_2 \in Y$ such that $Q_f(\ult_1, \ult,\ult_2)$. By \eqref{BT1CF}, $Q_f(\ult_2,\ult,\ult_1)$ which shows that $(Y \times \set{\ult} \times X) \cap Q_f \neq \z$, i.e., $\ult\in \poss{Q_f}(Y,X)$.

\smallskip

\eqref{ABT1g} Suppose $\ult\in\suff{S_g}(X,Y)$, that is $X\times\{\ult\}\times Y\subseteq S_g$. If $\ult_1\in Y$ and $\ult_2\in X$, then $S_g(\ult_2,\ult,\ult_1)$ and this with \eqref{BT1CFg} entails that $S_g(\ult_1,\ult,\ult_2)$. Since our choices were arbitrary, we have that $Y\times\{\ult\}\times X\subseteq S_g$, i.e., $\ult\in\suff{S_g}(Y,X)$.

\smallskip

\enlargethispage{30pt}

\eqref{ABT2}: Let $\ult\in Y \cap \poss{Q_f}(X,Z)$, i.e., $\ult \in Y$ and there are $\ult_1 \in X, \ult_2 \in Z$ such that $Q_f(\ult_1,\ult,\ult_2)$. By \eqref{BT2CF} we have $Q_f(\ult_1,\ult_1,\ult)$, so $\ult_1\in\poss{Q_f}(X,Y)$. Thus, $\ult_1\in X\cap\poss{Q_f}(X,Y)$. As $\ult_2$ is in $Z$, we obtain:
\[
\ult\in \poss{Q_f}(X\cap\poss{Q_f}(X,Y),Z)
\]
as required.


\smallskip

\eqref{ABT3}:  Suppose that $X,Y,Z \subseteq \Ult(A)$ and $\ult \in \poss{Q_f}(X, \suff{S_g}(X,-Y) \cap Y)$. Then, there are $\ult_1 \in X, \ult_2 \in  \suff{S_g}(X,-Y) \cap Y$ such that $Q_f(\ult_1, \ult, \ult_2)$. Now,
\begin{gather*}
\ult_2 \in  \suff{S_g}(X,-Y) \iff (X \times \set{\ult_2} \times -Y) \subseteq S_g.
\end{gather*}
Assume that $\ult \not\in Y$; then, $S_g(\ult_1,\ult_2,\ult)$ and \eqref{BT3CF} implies that $\ult = \ult_2 \in Y$, a contradiction.

\smallskip
For a weak betweenness algebra $\frA$ we need to show that \eqref{ABTW} holds in $\Em^{ps}(\frA)$. To this end, suppose that $\z \neq X \subseteq \Ult(A)$, and $\ult \in \suff{S_g}(X,X)$; then, $(X \times \set{\ult} \times X) \subseteq S_g$. Since $X \neq \z$, there is some $\ultV \in X$, and $S_g(\ultV,\ult,\ultV)$. \eqref{BTWCF} now implies $\ult = \ultV$, thus, $\ult \in X$.

\smallskip

Finally, for a strong betweenness algebra $\frA$ we need to prove \eqref{ABT2s} for $\Em^{ps}(\frA)$. Yet this is straightforward, since if $\emptyset\neq Y\subseteq\Ult(A)$ and $\ult_1\in X$, then by \eqref{BT2sCF} we have that $Q_f(\ult_1,\ult_1,\ult)$ for some $\ult\in Y$. This entails that $\ult_1\in\poss{Q_f}(X,Y)$.
\end{proof}

What we have proven in this section so far is the following: For any b-algebra\footnote{Everything that we write about here applies to weak and strong b-algebras as well.} $\frA\defeq\langle A,f,g\rangle$ its canonical frame $\Cf^{ps}(\frA)=\langle\Ult(A),Q_f,S_g\rangle$ behaves in such a way that the relations $Q_f$ and $S_g$ in a certain way <<simulate>> betweenness axioms. Furthermore, the standard Stone mapping embeds $\frA$ into $\Em^{ps}(\frA)=\langle 2^{\Ult(A)},\poss{Q_f},\suff{S_g}\rangle$. Nonetheless, neither $\Cf^{ps}(\frA)$ is necessarily a b-frame, nor $\Em^{ps}(\frA)$ is necessarily a complex b-algebra. Of course, if $\frA$ is a MIA, i.e., $Q_f=S_g$, then we can take the reduct of $\Cf^{ps}(\frA)$ to $\langle\Ult(A),Q_f\rangle$ which is a b-frame, and we can embed $\frA$ via Stone mapping into the complex b-algebra $\langle2^{\Ult(A)},\poss{Q_f},\suff{Q_f}\rangle$. The problem is---as we will see in the next section in Theorem~\ref{th:no-infinite-b-MIAS}---that there are no infinite b-algebras that are MIAs.

Unfortunately, the question whether for any given b-algebra $\frA$ there exists a b-frame $\frF\defeq\langle U,B\rangle$ such that $\frA$ can be embedded into $\Cm^{ps}(\frF)$ has a negative answer. This will be proven in the next section in the form of Example~\ref{ex:non-b-representable}.

\section{B-algebras and b-complex algebras}\label{sec:BetAlg-BetComplAlg}

In this section we will investigate the connections between b-complex algebras and (abstract) b-algebras and exemplify some instances in which they differ. Let us start with the following definitions:
\begin{definition}
The class of b-algebras is denoted by $\Abtw$.  An algebra $\frA \in \Abtw$ is a \emph{b-complex  algebra} or \emph{representable} if there exists a b-frame $\frF$ such that $\frA$ is isomorphic to a subalgebra of $\Cm^{ps}(\frF)$. The class of b-complex algebras is denoted by $\Cbtw$. As earlier, $\frA$ is a \emph{full} b-complex algebra when it is isomorphic to $\Cm^{ps}(\frF)$ of a b-frame $\frF$.
\QED\end{definition}

Let us briefly recall the situation.  For a b-algebra $\frA$ we start with operators $f,g$ which lead to ternary relations $Q_f$ and $S_g$ on the ultrafilter frame  which, in turn, lead to operators $\poss{Q_f}$ and $\suff{S_g}$, and the embedding of $\frA$ into $\Em^{ps}(\frA)$ is straightforward (Section \ref{sec:repAbtw}). This direction does not involve b-frames.

From frames to algebras we start with a single relation $B$ on $U$ which leads to a b-algebra on $2^U$ with  operators $\poss{B}, \suff{B}$ which, in turn, give us two relations $Q_{\poss{B}}$ and $S_{\suff{B}}$ which, in general, are not equal. One might ask into which one, if any, can we embed our $B$. The answer is: It does not matter. The reason for this is the fact that the mapping $k\colon U \to \Ult(\Cm^{ps}(\frF))$ of Theorem \ref{thm:appFrameRep} sends points of $U$ to principal ultrafilters of $2^U$, and  the relation $\set{\klam{k(x), k(y), k(z)}\mid x,y,z \in U}$ is isomorphic to $B$; it is contained in both $Q_f$  and $S_g$.  This holds for any 3-frame, so the observation is not particular to b-relations.

We first describe the b-algebras on the set of constants:

 \begin{example}\label{ex:QnotS}
 Suppose that $A \df \set{\zero,\one}$, and let $a,b \in A$. If $a =\zero$ or $b =\zero$, then the values of $f(a,b)$ and $g(a,b)$ are determined by the normality, respectively, co-normality of $f$, respectively, $g$; furthermore, $f(\one,\one) = \one$ by \eqref{ABT0}.

 If $g(\one,\one) = \one$, then the corresponding algebra $\frA_0$ is isomorphic to the smallest full b-complex algebra: Let $\frF \df \klam{U,B}$ where $U\defeq \set{x}$ and $B \df \set{\klam{x,x,x}}$. Then,
$$
\begin{array}{c|ccccc|cccc}
\poss{B} & \z & U &&& \suff{B} & \z & U \\ \hline
\z & \z & \z &&& \z & U & U \\
U & \z & U &&& U & U & U
\end{array}
$$
and $Q_{\poss{B}} = \set{\klam{\set{U}, \set{U}, \set{U}}} = S_{\suff{B}}$. This algebra is special in the sense that it is not a proper subalgebra of a b-algebra by Theorem~\ref{thm:ABTW+}(2).

Finally, let $g(\one,\one) = \zero$; then, $f$ and $g$ are as below:
 \[
 \begin{array}{c|ccccc|cc}
f & \zero & \one &&&g& \zero & \one    \\ \hline
\zero & \zero & \zero  &&& \zero & \one & \one \\
\one & \zero & \one &&& \one & \one & \zero
 \end{array}
 \]
It is easy to see that  $\frA_1\defeq\klam{A,f,g}$ is a strong b-algebra, and  $Q_f(\set{\one}, \set{\one}, \set{\one})$, but not $S_g(\set{\one}, \set{\one}, \set{\one})$ as $g[\{\one\}\times\{\one\}]\cap\{\one\}=\emptyset$. Obviously, $\frA_1$ is not isomorphic to $\Cm^{ps}(\frF)$, and so cannot be isomorphic to the full complex algebra of any b-frame.  On the other hand, $\frA_1$ is a b-complex algebra by Theorem~\ref{thm:CmNotSubalg} below.
\QED\end{example}

\begin{theorem}\label{thm:CmNotSubalg}
    Every b-algebra $\frA$ with at least four elements contains an isomorphic copy of the two-element algebra $\frA_1$ of Example~\ref{ex:QnotS} as a subalgebra.
\end{theorem}
\begin{proof}
Suppose that $\frA' \df \klam{A,f,g} \in \Abtw$  has at least four elements, and let $\two$ be the subalgebra of $\frA'$ generated by $\set{\zero,\one}$. Then, $\two$ is isomorphic to $\frA_0$ or $\frA_1$ of Example \ref{ex:QnotS}, and it is not isomorphic to $\frA_0$ since $\frA$ has at least four elements.
\end{proof}

We know from Theorem~\ref{th:full-complex-is-BTA} that the full PS-complex algebra of a b-frame is a b-algebra, i.e., that $\Cbtw \subseteq \Abtw$. The natural question arising at this point is whether the converse inclusion holds, that is, \emph{whether every b-algebra is a complex algebra.}\footnote{The question is motivated by a similar result for unary weak MIAs and the logic $K\tilde{~}$, see \cite[Important Lemma]{gpt87}, \cite[Theorem 8.5]{dot_mixed}.} As the following example shows, this is not necessarily the case.

\begin{example}\label{ex:non-b-representable}
    Let $A$ be the eight-element Boolean algebra with atoms $\set{a,b,c}$ and $f$ and $g$ be as below (for the brevity of presentation we write $xy$ instead of $x+y$):
 \[
 \begin{array}{c|ccccccccccc|cccccccc}
f & \zero & a& b&c&ab&ac&bc & \one &&&g& \zero& a& b&c&ab&ac&bc & \one    \\ \hline
\zero & \zero & \zero & \zero & \zero & \zero & \zero & \zero & \zero  &&& \zero & \one & \one& \one& \one& \one& \one& \one& \one \\
a & \zero & a & \one& \one& \one& \one& \one& \one &&& a & \one & a& a& \zero& a& \zero& \zero& \zero\\
b & \zero & \one & b& bc& \one& \one& bc& \one &&& b & \one & a& \zero& \zero& \zero& \zero& \zero& \zero\\
c & \zero & \one & bc& c& \one& \one& bc& \one &&& c & \one & \zero& \zero& \zero& \zero& \zero& \zero& \zero\\
ab & \zero & \one & \one& \one& \one& \one& \one& \one &&& ab & \one & a& \zero& \zero& \zero& \zero& \zero& \zero\\
ac & \zero & \one & \one& \one& \one& \one& \one& \one &&& ac & \one & \zero& \zero& \zero& \zero& \zero& \zero& \zero\\
bc & \zero & \one & bc& bc& \one& \one& bc& \one &&& bc & \one & \zero& \zero& \zero& \zero& \zero& \zero& \zero\\
\one & \zero & \one & \one& \one& \one& \one& \one& \one &&& \one & \one & \zero& \zero& \zero& \zero& \zero& \zero& \zero\\
 \end{array}
 \]
$\frA\defeq\langle A,f,g\rangle$ is a PS-algebra, and by routine calculations\footnote{Software assisted, if one wishes, e.g. by UACalc \citep{UACalc}.} we can check that it is a b-algebra. We will demonstrate that $\frA$ cannot be embedded into the full complex algebra of a b-frame, and therefore, it is not a b-complex algebra.
%

Assume towards a contradiction that it is. Then, there exists a b-frame $\frF =\klam{U,B}$ such that $\frA$ is isomorphic to a subalgebra $S$ of $\Cm^{ps}(\frF)$. Let $i\colon A\to S$ be an isomorphism of b-algebras. Thus, $i(a)\neq\emptyset$, and $\suff{B}(i(a),i(a))=i(a)$, that is $\suff{B}(i(a),i(a))\neq\emptyset$. It follows from Theorem~\ref{thm:ABTW+} that $i(a)$ is a singleton, i.e., there is $x\in U$ such that $i(a)=\{x\}$. On the other hand, we have $\poss{B}(i(a),i(c))=U$ and $\suff{B}(i(a),i(c))=\emptyset$. We will show that this is not possible because $x\in \suff{B}(i(a),i(c))$. Let $y\in i(c)$. Since $y\in \poss{B}(i(a),i(c))$, there exists $z\in i(c)$ such that $\klam{x,y,z}\in B$. By \eqref{BT2}, $\klam{x,x,y}\in B$. Thus, $i(a)\times\set{x}\times i(c)\subseteq B$ by arbitrariness of $y$, and we obtain a contradiction.\QED
\end{example}

The weak MIA axiom \eqref{g2f2} implies that $S_g \subseteq Q_f$, and one may ask under which conditions equality $Q_f = S_g$ holds, and, if so, whether $Q_f$ is a b-relation;  we shall answer this question below. In analogy to the unary case \eqref{MIA} we say that $\frA$ is a \emph{mixed algebra} (MIA), if $Q_f = S_g$, i.e., $f[\ult_1 \times \ult_3] \subseteq \ult_2$ \tiff $g[\ult_1 \times \ult_3] \cap \ult_2 \neq \z$.

Our next result implies that full finite b-complex algebras are MIAs:
\begin{proposition}\label{lem:Q=Sprinc}
Assume that $\ult_1, \ult_2,\ult_3$ are principal ultrafilters of $2^U$. Then $Q_{\poss{B}}(\ult_1, \ult_2,\ult_3)$ implies $S_{\suff{B}}(\ult_1, \ult_2,\ult_3)$. In particular, if $U$ is finite, then $Q_{\poss{B}}=S_{\suff{B}}$.
\end{proposition}
\begin{proof}
Let $\ult_1 = \ua{\set{x}}, \ult_2 = \ua{\set{y}}, \ult_3 = \ua{\set{z}}$ and $Q_{\poss{B}}(\upop\{x\},\upop\{y\},\upop\{z\})$. Then $\poss{B}[\upop\{x\} \times\upop\{z\}] \subseteq \upop\{y\}$, in particular, $\poss{B}(\set{x},\set{z}) \in \ua{\set{y}}$, which implies $B(x,y,z)$. It follows that $y \in \suff{B}(\set{x}, \set{z})$, hence, $\suff{B}(\set{x}, \set{z}) \in \ua{\set{y}}$, and therefore, $\suff{B}[\upop\{x\} \times\upop\{z\}] \cap \ua{\set{y}} \neq \z$, i.e., $S_{\suff{B}}(\upop\{x\},\upop\{y\},\upop\{z\})$.
\end{proof}

On the other hand, there are finite b-complex algebras in which the equality $Q_f=S_g$ fails such as the algebra $\frA_1$ of Example \ref{ex:QnotS}.
Our next example exhibits an infinite full complex b-algebra which is not a MIA. Theorem \ref{th:no-infinite-b-MIAS} below shows that this is no accident.

\begin{example}\label{ex:notMIA}
By Corollary~\ref{prop:B-from-poset}, the relation $B_{\leqslant}$ obtained from the natural ordering $\leqslant$ of $\omega$ is a betweenness relation, i.e., $\frN\defeq\klam{\omega,B_{\mathord{\leqslant}}}$ is a b-frame. Let $X,Y \subseteq \omega$  be infinite; then, both $X$ and $Y$ are cofinal, i.e., for every $n \in \omega$ there are $k \in X, m \in Y$ such that $n \leqslant k,m$.  Hence, $n \in \poss{B_{\leqslant}}(X,Y)$ \tiff $\min{X} \leqslant n$ or $\min{Y} \leqslant n$, and it follows that $\poss{B_{\leqslant}}(X,Y) = \ua{\min(X \cup Y)}$. In particular, $\poss{B_{\leqslant}}(X,Y)$ is cofinite, and therefore, $\poss{B_{\leqslant}}[\ult_1 \times \ult_2] \subseteq \ult$ for all free ultrafilters $\ult_1,\ult_2,\ult$ of $\omega$. On the other hand, $\suff{B_{\leqslant}}(X,Y) = \z$. Assume that $k \in \suff{B_{\leqslant}}(X,Y)$. Since $X,Y$ are cofinite, there are $n \in X, m \in Y$ such that $k<n,m$ which implies that $\klam{n,k,m} \not\in B$, a contradiction. Therefore $Q_{\poss{B_{\leqslant}}}\neq S_{\suff{B_{\leqslant}}}$, and the full complex algebra of $\frN$ is not a~MIA.

Let us point out an analogy holding between $\frN$ and the standard ultrafilter extension of the binary frame $\frN_2\defeq\langle\omega,\mathord{\leq}\rangle$. In case of the latter, principal ultrafilters correspond to the natural numbers, and free ultrafilters are clustered at infinity, in the sense that every free $\ult$ is accessible from every principal $\ult'$, and the accessibility relation is universal on the set of free ultrafilters. In case of $\frN$ the situation is somewhat analogous, since as we have seen $Q_{\poss{B_{\leqslant}}}$ is universal on free ultrafilters, and for any $n,m \in \omega$ such that $m\leqslant n$, and any free $\ult$ we have $Q_{\poss{B_{\leqslant}}}(\upop\{m\},\upop\{n\},\ult)$. Indeed, as we have seen above, $\poss{B_{\leqslant}}(X,Y) = \ua{\min(X \cup Y)}$, and since $m\in X\cup Y$, $\min(X \cup Y)\leqslant n$. Therefore, $n\in \ua{\min(X \cup Y)}$, as required. Thus $Q_{\poss{B_{\leqslant}}}$ may be viewed as a relation that puts a large cluster of free ultrafilters at infinity in case of the 3-frame $\frN$.
\QED
\end{example}

\begin{theorem}\label{th:no-infinite-b-MIAS}
    No infinite b-algebra is a MIA.
\end{theorem}
\begin{proof}
We will proceed in three steps:
\begin{enumerate}
\item \emph{No infinite full b-complex algebra is a MIA}: Let $\ult$ be a free ultrafilter on $2^U$. Then, $Q_{\poss{B}}(\ult,\ult,\ult)$ by  Lemma \ref{lem:cmp}. Assume that $S_{\suff{B}}(\ult,\ult,\ult)$. Then, there are $X,Y \in \ult$ such that $\suff{B}(X,Y) \in \ult$. Since $X \cap Y \in \ult$ and $\ult$ is a free ultrafilter, $X \cap Y$ is infinite. This contradicts Theorem~\ref{thm:ABTW+}.

\item \emph{No infinite subalgebra of a full b-complex algebra is a MIA}: Let $\mathfrak{S}\defeq\klam{ D,\poss{B}_D,\suff{B}_D}$ be an infinite subalgebra of $\Cm^{ps}(\frF)$, and $\ult$ be a free ultrafilter of $D$. Observing that $\Abtw$ is a universal class, we see that $\mathfrak{S}$ is a b-algebra, and therefore, $Q_{\poss{B}_D}(\ult,\ult,\ult)$, again by  Lemma \ref{lem:cmp}.

    Assume that $S_{\suff{B}_D}(\ult,\ult,\ult)$, i.e., $\suff{B}_D[\ult \times \ult] \cap \ult \neq \z$. Let $X,Y \in \ult$ be such that $\suff{B}_D(X,Y) \in \ult$; then, $\suff{B}_D(X,Y)\neq\emptyset$ and $X \cap Y$ is infinite. $\suff{B}_D(X,Y) = \suff{B}(X,Y)$, so again, we have a contradiction with Theorem~\ref{thm:ABTW+}.


\item \emph{No infinite b-algebra is a MIA}:     Suppose $\frA\defeq\klam{A,f,g}$ is an infinite MIA. Then $Q_f=S_g$, and the reduct of $\Cf^{ps}(\frA)$ to $\klam{\Ult(A),Q_f}$ is a b-frame by Lemma~\ref{lem:cmp}. Hence, $\Em^{ps}(\frA)$ is a full b-complex algebra, and by Theorem \ref{thm:repalg}, $\frA$ is isomorphic to a subalgebra of $\Em^{ps}(\frA)$. It follows from the previous results that $\frA$ is not a MIA.
\end{enumerate}
This concludes the proof.
\end{proof}

\pagebreak

The next example is an algebraic explanation of the non-definability result of Theorem \ref{thm:BnotP}.

\begin{example}\label{ex:Z}
Let $\frZ\defeq\langle\Intgr,B_{\leqslant}\rangle$ be as in Theorem \ref{thm:BnotP} and consider its full complex algebra $\Cm^{ps}(\frZ)$. Let $E$ and $O$ be the sets of even and odd integers, respectively. Then, $A \df \set{\z,E,O,\Intgr}$ is a Boolean subalgebra of $2^{\Intgr}$, and we shall show that it is a PS-subalgebra of $\Cm^{ps}(\frZ)$. If $x \in \Intgr$, then there are even and odd numbers below and above $x$ which shows that $\poss{B}(E,E) = \poss{B}(E,O) = \poss{B}(O,E) = \poss{B}(O,O) =\Intgr$. Furthermore, there are odd and even numbers strictly greater than $x$, and it follows that $\suff{B}(E,E) = \suff{B}(E,O) = \suff{B}(O,E) = \suff{B}(O,O) =\z$. Thus, $\frA\defeq\klam{A,\poss{B},\suff{B}}$ is a PS-subalgebra of $\Cm^{ps}(\frZ)$.

On the other hand consider the full complex algebra of the frame $\klam{U,R}$, where $R$ is the ternary universal relation on $U \df \set{w_0,w_1}$. There, $\suff{R}(X,Y) = U$ for all $X,Y \subseteq U$, contrasting the considerations above.

The algebra $\Cm^{ps}(\frZ)$ has yet one more interesting property. Since for any $X,Y\in 2^U$, $\suff{B}(X,Y)$ is always finite, if $\suff{B}[\ult_1\times\ult_2]\cap\ult\neq\emptyset$, $\ult$ must be a principal filter, i.e., for some $x\in U$, $\ult=\upop\{x\}$. Therefore, $Q_{\poss{B}}(\ult,\ult,\ult) $ and not $S_{\suff{B}}(\ult,\ult,\ult)$ for all free ultrafilters of $2^{\Intgr}$.

As can be seen from the proof above, the only property of $\langle\Intgr,\mathord{\leqslant}\rangle$ that was needed to obtain both conclusions was its order type, which is $\omega^\ast+\omega$. 
\QED\end{example}

\begin{example}
 In Example~\ref{ex:Z} we saw a full complex algebra for which the relation $S_{\suff{B}}$ is empty on the set of triples of free ultrafilters. We will show now that this does not have to be the case. To this end,  consider the strong b-frame $\frQ\defeq\langle\Rat,\Bleqs\rangle$ induced by the standard dense linear order $\leqslant$ on the set of rational numbers~$\Rat$. Consider three intervals $X\defeq(-\infty,0)$, $Y\defeq [0,1]$ and $Z\defeq(1,+\infty)$. Let $\ult_X,\ult_Y$ and $\ult_Z$ be free ultrafilters containing the respective intervals. In this case we have that $\suff{B}(X,Z)=Y$ and so $\suff{B}[\ult_X\times\ult_Z]\cap\ult_Y\neq\emptyset$. Generally, for any pair $\ult_1,\ult_2$ of free ultrafilters such that $(r,p)\in\ult_1$, $(q,s)\in\ult_2$ and $r<p<q<s$ there is a free ultrafilter $\ult$ such that $S_{\suff{B}}(\ult_1,\ult_2,\ult)$, i.e., any such ultrafilter which contains the open interval $(p,q)$.\QED
\end{example}

\section{Summary and further work}\label{sec:conclusion}

We have put forward---well justified---axioms for an algebraic treatment of reflexive betweenness relations. The resulting class $\Abtw$ of b-algebras turned out to have some <<good>> properties such as closure under canonical extensions, and some <<bad>> ones, too, such as non-b-representability. In the context of the latter, the most pressing questions now are the following:
\begin{enumerate}
    \item Is there for any b-algebra $\frA$ a 3-frame $\frF\defeq\klam{U,R}$ such that $\frA$ is embeddable into the complex PS-algebra of $\frF$? We know from Example~\ref{ex:non-b-representable} that $\frF$ cannot be a b-frame, but it may be possible that there is a larger class of 3-frames that can give us representability. Basically we ask if we can prove an analog of Theorem 8.5 from \cite{dot_mixed} or the \emph{Important Lemma} from \cite{gpt87}.
    \item Which, if any, axioms can we add to \eqref{ABT0}--\eqref{ABT3} and \eqref{g2f2} to obtain a subclass of $\Abtw$ that is b-representable?
\end{enumerate}

We also investigated properties of the possibility and sufficiency operators in the context of b-algebras, and although initially it seemed that our axioms say little about $g$, we were able to prove---especially in Section~\ref{sec:BetAlg}---quite a number of strong properties, some of them limiting in nature. Also, we knew that $f$ and $g$ should be bounded by certain axioms to cooperate fruitfully, but to our surprise, it turned out in Theorem~\ref{th:no-infinite-b-MIAS} that the cooperation cannot be too strong. Still, we believe that there are new connections to be discovered, which will result in further insights into algebraic aspects of betweenness.

What we entirely left out of this paper---but by no means neglected---are problems of axiomatizations of various subclasses of $\Abtw$ such as $\Cbtw$. These we are going to pursue in future installments of the work presented here.

\section*{Acknowledgements}

\begin{sloppypar}

This research was funded by the National Science Center (Poland), grant number 2020/39/B/HS1/00216, ``Logico-philosophical foundations of geometry and topology''.

For the purpose of Open Access, the authors have applied a CC-BY public copyright license to any Author Accepted Manuscript (AAM) version arising from this submission.

\end{sloppypar}

\bibliographystyle{plainnat}
\input{part_1.bbl}

\end{document}

\section{New result (temporary section)}

\begin{theorem}
    In every weak b-algebra $\frA$\/\textup{:}
    \[
        a\cdot b\neq\zero\neq g(a,b)\rarrow a\cdot b=g(a,b)\in\At(A)\,.
    \]
\end{theorem}
\begin{proof}
    Firstly, we show that in every weak b-algebra $\frA$, if $x\neq\zero\neq g(x,x)$, then $x$ is an atom and $g(x,x)=x$. To this end, suppose $x\neq\zero\neq g(x,x)$ and assume $a$ is a non-zero element such that $a<x$. By Theorem~\ref{thm:ABTW+} we have that $x\leq g(x,x)$. Using antitoness of $g$ and \eqref{ABTW} we obtain
    \[
    a<x\leq g(x,x)\leq g(a,a)\leq a\,,
    \]
which is a plain contradiction.

Secondly, assume $a\cdot b\neq\zero\neq g(a,b)$. Since $g(a,b)\leq g(a\cdot b,a\cdot b)$, we have that $g(a\cdot b,a\cdot b)\neq\zero$, and applying the result above, we see that $a\cdot b\in\At(A)$ and $a\cdot b=g(a\cdot b,a\cdot b)$. But:
    \[
        \zero\neq g(a,b)\leq g(a\cdot b,a\cdot b)=a\cdot b\,,
    \]
    so $g(a,b)=a\cdot b$, as required.
\end{proof}

%% file: macros.tex
\newtheorem{theorem}{Theorem}[section]
\newtheorem{lemma}[theorem]{Lemma}
\newtheorem{proposition}[theorem]{Proposition}
\newtheorem{corollary}[theorem]{Corollary}

\newtheoremstyle{mytheoremstyle} 
    {1em plus .2em minus .1em}                    
    {1em plus .2em minus .1em}                    
    {\rmfamily}                   
    {}                           
    {\bfseries}                   
    {.}                          
    {.5em}                       
    {}  

\theoremstyle{mytheoremstyle}

\newtheorem{definition}[theorem]{Definition}

\newtheorem{example}[theorem]{Example}

\newcommand\rarrow{\rightarrow}
\newcommand\larrow{\leftarrow}
\renewcommand\iff{\longleftrightarrow}

\newcommand\Rat{\varmathbb{Q}}

\newcommand\Intgr{\varmathbb{Z}}

\newcommand\iffdef{\;\mathrel{\mathord{:}\mathord{\longleftrightarrow}}\;}
\newcommand\defeq{\coloneqq} 
\newcommand\eqdef{\eqqcolon}

\newcommand{\zero}{\mathbf{0}} 
\newcommand{\one}{\mathbf{1}} 

\newcommand{\ult}{\mathscr{U}} 

\DeclareMathOperator{\Ult}{Ult} 

\DeclareMathOperator{\upop}{\uparrow} 

\newcommand\mfr{\mathfrak}
\newcommand\frA{\mfr{A}}

\newcommand\frF{\mfr{F}}

\newcommand\frN{\mfr{N}}

\newcommand\frQ{\mfr{Q}}

\newcommand\frZ{\mfr{Z}}

\newcommand\Rel{\mathrel{R}}

\newcommand\dftt{\mathtt{df}\,}



%% file: part_1.bbl
\providecommand{\noop}[1]{}